\documentclass[11pt, english, reqno]{amsart}

\usepackage[T1]{fontenc}
\usepackage{babel}
\usepackage{mathrsfs}
\usepackage{amsthm}
\makeindex

\usepackage[dvips,top=3.8cm,left=4cm,right=3.8cm,
foot=3.8cm,bottom=4.0cm]{geometry}
\usepackage{amsfonts,amssymb,amsmath,amsthm, booktabs, 
latexsym}
\usepackage{centernot}
\usepackage{tikz-cd}
\usepackage{bbm}

\usepackage{enumerate}
\usepackage{color}

\newtheorem{theorem}{Theorem}[section]
\newtheorem{lemma}[theorem]{Lemma}
\newtheorem{proposition}[theorem]{Proposition}
\newtheorem{corollary}[theorem]{Corollary}
\theoremstyle{definition}
\newtheorem{definition}[theorem]{Definition}

\theoremstyle{remark}
\newtheorem{remark}[theorem]{Remark}

\numberwithin{equation}{section}

\numberwithin{equation}{section}
\newcommand{\be}{\begin{equation}}
\newcommand{\ee}{\end{equation}}
\newcommand{\ba}{\begin{aligned}}
\newcommand{\ea}{\end{aligned}}

\newcommand{\N}{{\mathbb N}}
\newcommand{\Z}{{\mathbb Z}}

\def\csi1{\circ\sigma^{-1}}

\def\mc{\mathcal}

\newcommand{\B}{{\mathcal B}}

\begin{document}

\title[Cohomology of hyperfinite Borel actions]{Cohomology of hyperfinite Borel actions}

\author{Sergey Bezuglyi}
\address{Department of Mathematics, University of Iowa, Iowa City,
52242 IA, USA}
\email{sergii-bezuglyi@uiowa.edu}
\email{shrey-sanadhya@uiowa.edu}

\author{Shrey Sanadhya}



\begin{abstract}
We study cocycles of countable groups $\Gamma$ of Borel
 automorphisms of a
 standard Borel space $(X, \B)$ taking values in a locally compact second
countable group $G$. We prove that for a hyperfinite group $\Gamma$
  the subgroup of coboundaries is dense in the 
group of cocycles. We describe all Borel cocycles of the $2$-odometer
and show that  any such cocycle  is cohomologous to a cocycle with 
values in a countable 
dense subgroup $H$ of $G$. We also provide a Borel version of 
Gottschalk-Hedlund theorem.  
 
\end{abstract}

\maketitle

\tableofcontents
\section{ Introduction}\label{sec intro}

Let $\Gamma$ be a countable
group of Borel automorphisms of a standard Borel space $(X, \B)$ and
$G$ an abelian locally compact second countable (l.c.s.c.) group.
A Borel map $\alpha : \Gamma \times X \to G$ is called a 
\textit{cocycle} if it satisfies the so called \textit{cocycle identity} for all 
$(\gamma, x)$:
\be\label{eq-cocycle def}
\alpha(\gamma_1\gamma_2, x) = \alpha(\gamma_1, \gamma_2 x) 
+ \alpha(\gamma_2, x),\ \ \alpha(\mathbbm 1, x) = 0,
\ee
where $\mathbbm 1$ is the identity map and $0\in G$. 
A cocycle $\alpha(\gamma, x)$ is called a \textit{coboundary} if there 
exists a Borel function $f : X \to G$ such that $\alpha(\gamma, x) =
f(\gamma x) - f(x)$. Two cocycles, $\alpha$ and $\beta$, are called 
\textit{cohomologous} if $\alpha - \beta$ is a coboundary. 
The set  $Z^1(\Gamma\times X, G)$ of all cocycles is an abelian group,
and coboundaries $B^1(\Gamma \times X, G)$  form a subgroup of
 $Z^1(\Gamma\times X, G)$.

Cocycles play an important role in ergodic theory. They are studied up to 
null sets: if $\Gamma$ is a 
countable group of non-singular automorphisms of a standard measure
space $(X, \B, \mu)$, then relation \eqref{eq-cocycle def} must be true 
$\mu$-a.e.  Cocycles are widely used in
the theory of orbit equivalence of dynamical systems and in various
 constructions (e.g., skew product) helping to 
classify dynamical systems and clarify the properties of automorphism
groups of a measure space. They are also one of the central tools in 
the representation theory, theory of groupoids, classification of ergodic 
actions of amenable and non-amenable groups, etc.  Understanding the
 structure of cocycles for a \textit{hyperfinite automorphism group }
 (it is a group
 which is orbit equivalent to a single transformation) led to a
more detailed classification than orbit equivalence 
\cite{BezuglyiGolodets_1991}, \cite{GolodetsSinelshchikov_1987}, 
\cite{FeldmanSutherlandZimmer1989}, \cite{GolSinelsh1994}, 
\cite{Hamachi_2000}.  We give here 
several principal references which include some pioneering works of 
Moore \cite{Moore_1970}, Ramsay \cite{Ramsay_1971},
\cite{FeldmanMoore_1977},  
Schmidt \cite{Schmidt_1977, Schmidt_1990}, Zimmer 
\cite{Zimmer_1984}. (A more detailed list of papers devoted to 
cocycles is too long to mention all crucial contributions 
to the theory of cocycles.)   

It is well known that there are impressive parallels between ergodic
theory and Borel dynamics, though the fact that, for a Borel dynamical 
system, there is no prescribed measure on the underlying space makes these two theories essentially different. 
 In this paper, we prove several results about cocycles in the context
of Borel dynamics. They are motivated by the existing counterparts in
the framework of ergodic theory. There are many important problems 
in dynamics involving Borel cocycles that deserve to be studied. For 
example, it would be 
interesting to find out whether the notion of a ratio set makes sense for
Borel cocycles taking values in l.c.s.c. (abelian) groups.  Another 
application of cocycles might be related to
the study of a Borel version of Mackey range. These concepts    
are extremely important for the classification of automorphism groups
in ergodic theory. Remark that these and other results in ergodic theory 
remain true for 
nonabelian groups, in general. In the case of Borel cocycles, even abelian 
case is not well understood. We hope to contribute to the formulated 
problems in further works. In this paper, we focused on cocycles with 
values in abelian groups.  It is worth mentioning that various properties of
Borel cocycles were  considered in the  papers  \cite{Becker_2013},
\cite{ConzeRaugi_2009}, \cite{Danilenko1998}, 
 \cite{FeldmanMoore_1977}, \cite{Miller_2006}, \cite{Miller_2008},
 and some others. 

We fix the main setting for the paper: $\Gamma$ is a hyperfinite free
countable group of Borel automorphisms on a standard Borel space 
$(X, \B)$ and   $\alpha \in Z^1(\Gamma \times X, G)$ is a cocycle 
of $\Gamma$ with values in an abelian l.c.s.c. 
group $G$. In this setting, the following results are proved:
(i) we introduce a  topology on the space of
 Borel functions (which is an analogue of the convergence in 
measure topology) and prove that the set of coboundaries is dense in
the set of all cocycles;  (ii) using an exact formula that describes 
cocycles over an odometer,
we prove that every cocycle is cohomologous to a cocycle with values
in a dense countable subgroup; (iii) we give a criterion (a version
of Gottschalk-Hedlund theorem) for a cocycle with values in $G$ 
to be a coboundary.

The  study of Borel cocycles is mostly motivated by the theory of 
orbit equivalence of groups of Borel automorphisms. The property of
orbit equivalence for groups of 
Borel automorphisms is equivalent to isomorphism of the corresponding 
equivalence relations generated by orbits. The notion of a countable 
Borel equivalence relation (CBER)  has been  extensively studied in the
descriptive  set theory and Borel dynamics. This concept has many
 applications in other adjacent areas.  We refer to 
\cite{BeckerKechris_1996}, \cite{DoughertyJacksonKechris_1994},  
\cite{JacksonKechrisLouveau_2002}, \cite{Hjorth_2000}, 
\cite{Kechris_1995}, \cite{KechrisMiller_2004}, \cite{Nadkarni_2013},
\cite{Varadarajan_1963}, and 
\cite{Weiss_1984}, where the reader can find connections  of orbit 
equivalence theory with the descriptive set theory and further references.

Our main results about cocycles are of the following nature. Firstly, 
it is not hard to see that orbit equivalent groups of Borel automorphisms 
have isomorphic groups of cocycles and coboundaries and therefore 
the cohomology groups. This means that the study of cocycles is 
naturally reduced to the the case when cocycles are considered on some
``model'' CBERs. In this connection, two types of dynamical systems  
are of crucial importance: odometers and shifts. The classification 
of hyperfinite CBERs up to isomorphism was a significant achievement 
 due to Dougherty-Jackson-Kechris
\cite{DoughertyJacksonKechris_1994}. They proved that the complete
 invariant of isomorphism of hyperfinite CBERs is the cardinality of the 
set of invariant measures. Odometers represent CBERs with a unique 
probability invariant measure. They are also the main ingredient for 
the constructions of CBERs
with finite (or countable) set of probability ergodic invariant measures. 
In Section 
\ref{sec 5}, we give an explicit formula for cocycles of the 2-odometer.
Our proof follows the approach used in \cite{Golodec_1969} and
\cite{GolodetsSinelshchikov_1987}  for 
measurable dynamical systems.  

Another key result about cocycles in ergodic theory states that
 coboundaries  of a non-singular group of automorphisms 
 $\Gamma \subset Aut(X, \B, \mu)$  are dense in the group of all
 cocycles if  $\Gamma$ is hyperfinite,
see e.g., \cite{ParthasarathySchmidt_1977} and 
\cite{Schmidt_1990} for a proof. Here the set 
 $Z^1(\Gamma\times X, G)$ is endowed with the topology of 
 convergence in measure. In Borel dynamics we do not have  a 
 prescribed measure on $(X, \B)$. Hence, to define an analogue of
the topology of convergence in measure, we have to work with all Borel
probability measures. (Our approach is similar to that used in 
\cite{BezuglyiDooleyKwiatkowski_2006} where 
an analogue of the uniform topology on $Aut(X, \B)$ was defined). 
In Sections 
\ref{sect topologies} and  \ref{Sec Density}, we consider topological 
properties of $Z^1(\Gamma\times X, G)$ and prove that coboundaries
are dense in $Z^1(\Gamma\times X, G)$ if $\Gamma$ is hyperfinite. 

 Gottschalk and Hedlund (see 
\cite[Theorem 4.11]{Gottschalk_Hedlund_1955}) provided a criterion for 
determining when a bounded cocycle of a minimal homeomorphism of 
compact space is a coboundary.  It was extended to minimal 
homeomorphisms of non-compact topological space in 
\cite{Browder_1958}. It is a well know fact that every Borel 
automorphism admits a continuous model, i.e., it is Borel isomorphic 
to a homeomorphism of a Polish space. Using this model we extend the
Gottschalk-Hedlund theorem to bounded Borel cocycles (taking value in 
an abelian l.c.s.c group) of minimal homeomorphisms of Polish space 
(see Theorem $\ref{thm GH}$).

The \textit{outline} of the paper is as follows. In Section 
\ref{sec prelim}, we
provide basic definitions and preliminary results about groups of
Borel automorphisms and cocycles.  In Section \ref{sect topologies}, 
we define a topology on the space of $G$-valued functions and discuss 
properties of this topology. We show 
that the group of cocycles of a hyperfinite group of automorphisms 
is a separable Hausdorff topological group. In 
Section \ref{Sec Density}, we prove the main result stating that for a
hyperfinite Borel action the subgroup of coboundaries is dense in the group 
of cocycles with respect to the topology defined in Section 
\ref{sect topologies}. We study cocycles of the $2$-odometer in Section 
\ref{sec 5}. In Section $\ref{sec 6}$, we prove the Borel version of 
Gottschalk-Hedlund Theorem.

\textit{Notation and Terminology:} Here are a few remarks about 
the exposition of our results in this paper. Firstly, we prefer to use the
 terminology
 which is traditional for dynamical systems in ergodic theory. This means 
that our principal objects are countable groups of Borel automorphisms 
not equivalence relations. 
But we also use the language of CBERs when it is convenient. Secondly,
we are aware that some results can be reproved for cocycles 
with values in non-abelian  l.c.s.c. groups, for example, those in Section 
\ref{sec 5}. Meantime, we will work with abelian groups in this section for consistency. The case of non-abelian groups deserves a separate study.

 Throughout the paper, we use the following {\bf notation}:

\begin{itemize}

\item $(X,\mathcal{B})$ is a standard Borel space with the 
$\sigma$-algebra of Borel sets $\mathcal{B}= \mathcal{B}(X)$.

\item A one-to-one Borel map $T$ of the space $(X,\mathcal{B})$ onto itself 
is called a Borel automorphism of $X$. In this paper the term 
"automorphism"  means a Borel automorphism of $(X,\mathcal{B})$.

\item $Aut(X,\mathcal{B})$ is the group of all Borel automorphisms of 
$X$ with the identity map ${\mathbb I} \in Aut(X, \mathcal{B})$.

\item A countable subgroup $\Gamma $ of  $Aut(X,\mathcal{B})$ is
called a group of Borel automorphisms. The full group generated by
$\Gamma$ is denoted by $[\Gamma]$. 

\item $\mathcal{M}_1(X)$ is the set of all Borel probability measures on 
$(X,\mathcal{B})$. 

\item $E(S,T)=\{ x\in X \ |\ Tx\ne Sx\} \cup \{x\in X \ |\ T^{-1}x\ne 
S^{-1}x\}$ where $S,T \in$ $Aut(X,\mathcal{B})$.

\end{itemize}

\section{Preliminaries}\label{sec prelim}

In this section we provide the basic definitions from Borel dynamics and descriptive 
set  theory.

\subsection{Automorphisms of standard Borel space} Let $X$ denote a separable completely metrizable space (also known as a \textit{Polish space}), and 
let $\mathcal{B}$ be the $\sigma$-algebra generated by the open sets in 
$X$. Then the pair $(X, \mathcal{B})$ is called a \textit{standard Borel space}.

A countable subgroup $\Gamma$  of 
$ Aut(X, \mathcal{B})$ is called a \textit{Borel automorphism group}. In 
this paper we focus only on countable Borel automorphism groups. Let $G$ be a countable group with  identity $e$. A \textit{Borel action} of the group $G$ on $(X, \mathcal B)$ is a group homomorphism $\rho : g \rightarrow \rho_g : G
 \rightarrow Aut(X, \mathcal{B})$. In other words, for each $g\in G$, $\rho_g : X \rightarrow X$ is a Borel automorphism such that 
(i) $\rho_{gh}(x)=\rho_g(\rho_h (x))$ for every $h \in G$ and  (ii) $\rho_{e}(x)=x$ for every $x
\in X$. Clearly, 
$\rho(G) = \{\rho_g : g \in G\} \subset Aut (X,\mathcal{B})$ is a countable 
subgroup. If, for some $x \in X$, the relation  $\rho_g(x) = x$  implies $g = e$, then $\rho$ is called a \textit{free action} of $G$. In this case, the group homomorphism $\rho$ is injective. We note that every Borel automorphism $T\in Aut(X, \mathcal B)$ defines a Borel  action of the group $\mathbb Z$ by the formula $\mathbb  Z \ni n \mapsto T^n \in Aut(X, \mathcal B)$. 
 \\

\textit{Countable Borel equivalence relation (CBER)}: An equivalence relation $E$ on $(X, \mathcal{B})$ is called Borel if it is a Borel subset of the product
space $E \subset X \times X$, where $X \times X$ is equipped with the Borel $\sigma$-algebra $\mathcal{B} \times \mathcal{B}$. It is called countable if every equivalence class 
$[x]_E := \{y \in X : (x,y) \in E\}$ is countable for all $x \in X$. If $C$ is 
a Borel set, then $[C]_E$ denotes the saturation of $C$ with respect to
the equivalence relation $E$, i.e., $[C]_E$ contains the entire class $[x]_E$
for every $x\in C$. 

For a countable subgroup $\Gamma$  of $ Aut(X, \mathcal{B})$, we denote 
$$
    E_X(\Gamma) = \{(x,y) \in X \times X: x=\gamma y
    \  \mbox{for\ some}\ \gamma \in \Gamma\}.
    $$
Then $E_X(\Gamma)$ is called the \textit{orbit equivalence relation} 
generated by the group $\Gamma$ on $X$. Clearly, 
$E_X(\Gamma)$ is a CBER. An equivalence relation $E$ is called 
\textit{aperiodic} if every $E$-class $[x]_E$ is countably infinite. 
In contrast, finite $E$-classes will be called \textit{periodic.} Similarly,
a Borel automorphism $P$ is called periodic at a point $x$ if there exists $k 
\in \mathbb N$ such that $P^kx = x$. The least such $k$ is called the period
of $P$ at $x$.

 It turns out that all CBER's are generated by group automorphisms. 

\begin{theorem} [Feldman--Moore \cite{FeldmanMoore_1977}] \label{FM}  Let $E$ be a
 countable Borel equivalence relation on a standard Borel space $(X,
 \mathcal{B})$. 
Then there is a countable group $\Gamma$ of Borel automorphisms of 
$(X,\mathcal{B})$ such that $E = E_X(\Gamma)$.
\end{theorem}

A  Borel set $B$ is a \textit{complete section} for an equivalence relation
$E$ on $(X, \mathcal B)$ if it intersects every $E$-class, i.e., $[B]_E = X$. 
If a complete section intersects each $E$-class exactly once then it is called 
a \textit{Borel transversal}. An equivalence relation $E$ which admits a Borel
transversal is called \textit{smooth.} Equivalently, one can say that 
an equivalence relation  $E$ on a standard Borel space $(X, \B)$ is 
 \textit{smooth} if there is a Borel function $f: X \rightarrow Y$, where $Y$ 
is a standard Borel space, such that $(x,y) \in E \Longleftrightarrow f(x)= 
f(y)$. 
We remark that in this paper we will deal only with \textit{non-smooth} 
CBERs. See \cite{Kechris_2019} for a survey of the state of the art in the theory of 
countable Borel equivalence relations.

\begin{definition}\label{def set of sections} Let $\Gamma$ be a countable  
automorphism group acting on $(X, \B)$. We will denote by 
$\mathcal{C}_\Gamma$ the collection of Borel subsets $C$ such that 
 $C$ and $X \setminus C$ both are complete sections for $E_X(\Gamma)$.
\end{definition}

\textit{Full group of automorphisms.} 
For a countable subgroup $\Gamma$ of $ Aut(X, \mathcal{B})$, we denote 
by $\Gamma x$ the orbit  $\{\gamma x : \gamma \in \Gamma \}$  of $x$ 
with  respect to $\Gamma$. We say that $\Gamma$ is a \textit{free} group
of automorphisms  if $\gamma x \neq x$ for every $\gamma \neq e$ and 
$ x \in X$.

The set 
$$
 [\Gamma] = \{R \in Aut (X,\mathcal{B}) : Rx \in \Gamma x, \ \forall x\in X \}
$$
is called the {\it full group of automorphisms} generated by $\Gamma$. 
The full group generated  by a single automorphism $T \in Aut (X,
 \mathcal{B})$ is denoted by $[T]$. 

Let $\Gamma \subset  Aut (X, \mathcal{B})$ be a freely acting group of 
automorphisms of a standard Borel space $(X, \B)$. Then, for 
every $R \in [\Gamma]$, there exists a Borel function $\gamma_R : X \to 
\Gamma$ such that $R x = \gamma_R (x)x,\ x\in X $. It follows that 
 every $R \in [\Gamma]$ defines a countable partition of $X$ into Borel 
 sets $A_\gamma = \{ x\in X: \gamma_R(x) =\gamma\},\ \gamma\in 
 \Gamma$. Conversely, if $\{A_\gamma\}$ is a Borel partition of $X$, such that 
 $\{\gamma A_\gamma\}$ also constitutes a Borel partition of $X$, then the 
 map $R x =\gamma x$, $x \in A_\gamma$, defines an element of 
 $[\Gamma]$. In case when $\Gamma$ is generated by a single 
 automorphism $T$, the same construction holds, and each $R$ from $[T]$
is represented in terms of piecewise constant function $x \mapsto n_R(x)$.

A countable subgroup $\Gamma$ of $Aut (X, \mathcal{B})$ is called 
\textit{hyperfinite} if $\Gamma x = \bigcup _{i=1} ^\infty \Gamma_i x$ for 
every $x \in X$, where each $\Gamma_i$ is a finite subgroup of 
$[\Gamma]$ and $\Gamma_i \subset \Gamma_{i+1}$ for all $i$.
Equivalently, a countable Borel equivalence relation $E$ is called 
\textit{hyperfinite} if $E = \bigcup_{n} E_n $ where $E_n \subset E_{n+1}$
for all $n$, where each $E_n$ is a finite Borel sub-equivalence relation of 
$E$. 

Let $\Gamma_i$ be  a countable subgroup of 
$Aut(X_i, \mathcal{B}_i)$, $i =1,2$. The groups $\Gamma_1$ and 
$\Gamma_2$ are called \textit{orbit equivalent} (denoted \textit{o.e.}) if 
there exists a Borel isomorphism $\varphi: (X_1,\mathcal{B}_1) \rightarrow 
(X_2, \mathcal{B}_2)$ such that $\varphi \Gamma_1 x = \Gamma_2 
\varphi x$, $\forall x \in X_1$. Equivalently,
$$
\varphi [\Gamma_1]\varphi^{-1} = [\Gamma_2].
$$
If $E_X(\Gamma)$ is the equivalence relation generated by a free action of 
$\Gamma$, then the orbit equivalence of $\Gamma_1$ and $\Gamma_2$ is
equivalent to the isomorphism of $E_{X_1}(\Gamma_1)$ and 
$E_{X_2}(\Gamma_2)$. 

We  refer readers to \cite{DoughertyJacksonKechris_1994} for the classification of hyperfinite aperiodic CBER 
with respect to orbit equivalence. 

\begin{theorem}  [Slaman-Steel \cite{SlamanSteel_1988}, Weiss \cite{Weiss_1984}]\label{thm hyperfinite}
Suppose $E$ is a CBER. The following are equivalent:

$1.$ $E$ is hyperfinite.

$2.$ $E$ is generated by a Borel $\mathbb{Z}$-action.

\end{theorem} 

Below we recall the definition of an \textit{odometer} (known also 
as an \textit{adding machine}). There are many papers 
devoted to odometers and their generalizations. We refer the interested
 reader to \cite{GrigorchukNekrashevichSushchanski_2000} and 
 \cite{Nekrashevych_2005} for detailed discussion.

\begin{definition}\label{def odo} Let $\{p_n\}_{n=0}^{\infty}$ be a 
sequence of integers such that $p_n \geq 2$ for each $n$. Let $\Omega =
 \underset{n=0}{\overset{\infty}{\prod}} \{0,...,p_n -1\}$ be equipped with
product discrete topology. Then $\Omega$ is a Cantor set. We define 
$S: \Omega \rightarrow \Omega$ as follows: $S(p_0 -1, p_1 -1,...) = 
(0,0,...)$, and for any other $x \in \Omega$, find the least $k$ such that 
$x_k \neq p_k - 1$ and put $S(x) = (0,0,...,0,x_k +1, x_{k+1}, x_{k+2},...)$. A 
Borel automorphism $T$ is called an \textit{odometer} if it is Borel 
isomorphic to some $S$.  An odometer $S$ is called the $2$-\textit{adic 
odometer}, if $p_n = 2$ for each $n \in \N_0$. In section $\ref{sec 5}$ we 
will work with the $2$-adic odometer. For brevity, we will call it the $2$-
odometer. 

\end{definition}

\subsection{Cocycles of Borel automorphism group}

As above, let $\Gamma$ be a countable subgroup of $Aut(X,\mathcal{B})$ 
acting freely, and let  $G$ denote a locally compact second countable 
abelian group with identity $0$ (we will use the additive group operation).
We remark that the assumption that   $G$ is an abelian  group is made for 
convenience and can be dropped in the following definitions.

\begin{definition}\label{def cocycle} A Borel function $a: \Gamma 
\times X 
\rightarrow G$ is called a \textit{cocycle} over $\Gamma$ if for any 
elements $\gamma_1,\gamma_2 \in \Gamma$ and all  $x \in X$ 
\begin{equation}\label{eq cocyc 1}
a(\gamma_1 \gamma_2,x) = a(\gamma_1, \gamma_2 x) + a(\gamma_2, x)
\end{equation}
and
\begin{equation}\label{eq cocyc 2}
    a(\mathbbm 1, x)= 0
\end{equation}
where $\mathbbm 1$ denotes the identity map.
The set of all cocycles of $\Gamma$ is denoted by $Z^1(\Gamma 
\times X, G)$. 

A cocycle $a: \Gamma \times X \rightarrow G$ is called a 
\textit{coboundary} if there exists a Borel function $c: X \rightarrow G $ 
such that
\begin{equation}\label{eq cbd 1}
a(\gamma, x) = c(\gamma x)- c(x) , \quad \forall \gamma \in \Gamma,   
    \forall x \in X.
\end{equation}
The set of all coboundaries of $\Gamma$ is denoted by $B^1(\Gamma 
\times X, G)$. 

Cocycles $a_1$, $a_2 : \Gamma \times X  \rightarrow G$ are called 
\textit{cohomologous} ($a_1 \sim a_2 $) if their difference is a coboundary, 
i.e., if there exists a Borel function $ c : X \rightarrow G $, such that
\begin{equation}\label{eq cohom 1}
     a_1(\gamma, x) = c(\gamma x)+ a_2(\gamma, x) -c(x).
\end{equation}

\end{definition}

 Sometimes it is useful to define cocycles over an equivalence relation as 
 described below.

\begin{definition}\label{def orbital cocycle} Let $E$ be a CBER. A Borel 
function $u: E \rightarrow G$ is an \textit{orbital cocycle} over $E$ if for 
every $(x,y), (y,z), (x,z) \in E$  
\begin{equation}\label{eq orb c1}
    u(x,z)= u(x,y)+u(y,z).
\end{equation}
An orbital cocycle is a \textit{coboundary} if there exists a Borel function $c: 
X \rightarrow G$ such that for $(x,y) \in E$
\begin{equation}\label{eq Orb cbd}
    u(x,y)= c(x)- c(y).
\end{equation}
As before, two orbital cocycles are \textit{cohomologous} if their difference 
is a coboundary. 
\end{definition}

\begin{remark}\label{R_1} Let $\Gamma$ be a freely acting countable 
 group of automorphisms. Given 
any cocycle $a \in Z^1(\Gamma \times X, G)$, define a function $u_a: 
E_X(\Gamma) \rightarrow G$ by the following rule: for any pair $(y, x)\in 
E_X(\Gamma)$ determine unique $\gamma \in \Gamma$ such that 
$y = \gamma x$ and then set 
\begin{equation}\label{eq orb cohom}
    u_a(y,x) = a(\gamma,x). 
\end{equation}
Since $\Gamma$ is free, $u_a$ is well-defined. It is clear that $u_a$ 
satisfies $(\ref{eq orb c1})$, hence it is an orbital cocycle. Moreover, 
$u_a$ is a coboundary if and only if $a$ is a coboundary. 
Conversely, every  orbital cocycle of $E_X(\Gamma)$ defines a cocycle 
of $\Gamma$.
\end{remark}

\begin{remark}\label{R_a} Let $T$ be an automorphism of $(X, \B)$ 
which determines an action of the group $\mathbb Z$.  Every Borel function 
 $f : X \to G$ with values in the group $G$ defines a cocycle 
$a: \mathbb{Z} \times X \rightarrow G$ by the formula   
\begin{equation}\label{eq z-cocycle}
    a(j,x) = \begin{cases}
f(x)+ f(Tx)+...+f(T^{j-1}x), & \quad j\geq 1 \\ 
0,  & \quad j = 0\\ 
-f(T^{-1}x) - f(T^{-2}x) - ... -f(T^{j}x), & \quad j\leq-1,
\end{cases}    
\end{equation}

Conversely, if  $a : \mathbb Z \times X \to G$ is a cocycle of the group 
$\{T^n, n \in \mathbb{Z}\}$, then it is completely determined by the
function $f(x) = a (1, x)$. Moreover, the properties of the cocycle $a(j, x)$
are represented in terms of the function $f$.
\end{remark}

\begin{remark} If $a : \Gamma \times X \to G$ is a cocycle of a freely 
acting countable group of automorphisms $\Gamma$, then it can be 
extended  to a cocycle $\widehat a$ over the full group $[\Gamma]$.
 Indeed, for 
$R \in [\Gamma]$ take the uniquely determined function $x \mapsto
 \gamma(x)$ such that $Rx = \gamma(x) x$. Then we set 
$$
\widehat a(R, x) = a(\gamma(x), x), \quad x\in X.
$$
It can be easily seen that $\widehat a$ coincides with $a$ on elements of
the group $\Gamma$, and $\widehat{a}$ satisfies the cocycle identity
\eqref{eq cocyc 1} and \eqref{eq cocyc 2}.
\end{remark}

\subsection{Topologies on the group $Aut(X, \B)$}
We will need the notion of convergence of a sequence of Borel 
automorphisms. Recall that 
several topologies on $Aut(X,\mathcal{B})$ were defined and studied in
\cite{BezuglyiDooleyKwiatkowski_2006}. We will work with the so called uniform topology $\tau$
whose origin lies in ergodic theory (see Introduction for the definition of
 $\mathcal M_1(X)$ and $E(S,T)$).

\begin{definition}\label{def topology} The \textit{uniform topology} 
$\tau$, on $Aut(X,\mathcal{B})$ is defined by the base of neighborhood $
\mathcal{V} = \{V(T; \mu_1,...,\mu_n; \epsilon)\}$ where, $ T\in Aut(X,
\mathcal{B}) $, $ \mu_1,...,\mu_n \in {\mathcal{M}}_1(X)$, $\epsilon>0 $, 
and
\begin{equation}\label{eq top t}
V(T; \mu_1,...,\mu_n; \epsilon)= \{ S\in Aut(X,\mathcal{B})\ |\ \mu_i(E(S,T))<\epsilon ,\ i=1,...,n\}.
\end{equation}

\end{definition}

\begin{remark}\label{top} It can be seen that  $(Aut(X,\mathcal{B}), \tau)$ 
is a Hausdroff, topological group. It is also relevant to mention that topology 
$\tau$ coincides with the topology $\tau'$, which is defined by the base of 
neighborhood $\mathcal{V'} = \{V'(T;\mu_1,...,\mu_n; \epsilon)\}$ where, 
$ T\in Aut(X,\mathcal{B}) $, $ \mu_1,...,\mu_n \in {\mathcal{M}}_1(X)$, $
\epsilon>0 $, and
\begin{equation}\label{eq top t'}
V'(T; \mu_1,...,\mu_n; \epsilon )= \{ S\in Aut(X,\mathcal{B}) \ |\ \sup_{F\in {\mathcal{B}}}\mu_i(TF\ \Delta \ SF)<\epsilon,\ i=1,...,n\}.
\end{equation}

\end{remark}

\section{Topologies on the space of cocycles}\label{sect topologies}

For a standard Borel space $(X, \B)$ and an abelian l.c.s.c. group $G$, 
we denote by  $\mathcal{F}(X, G)$  the set 
of  Borel functions $f:X \rightarrow G $. Clearly, this set is an abelian group under 
pointwise addition of functions. We will write simply $\mathcal{F}$ when $X$ and 
$G$ are understood. Since $G$ is  metrizable, we will 
denote by $\left| \ \cdot \ \right| $ a translation invariant metric on $G$
 compatible with the topology on $G$. 

In this section we will define and study 
topologies on $\mathcal{F}(X, G)$ which are analogous to the topology of  
convergence in measure. For a countable group of Borel 
automorphisms $\Gamma \subset Aut(X,\B)$, we will  consider the
subgroups  of  cocycles and coboundaries in 
$\mathcal{F}(X, G)$. Our goal
is to show that, for a hyperfinite group $\Gamma$, coboundaries form 
a dense subgroup in the group of all cocycles.

\begin{remark}\label{orbit eq} Let $\Gamma$ be a hyperfinite countable 
subgroup of $Aut(X,\mathcal{B})$. Without loss of generality, we can assume 
that $\Gamma$ acts freely. Then $\Gamma$ is orbit equivalent to a Borel 
$\mathbb{Z}$-action, i.e., there exists an automorphism 
$T \in Aut(X, \mathcal{B})$, such that the orbits $\Gamma(x)$ coincide 
with those of the group $\{T^n x, n \in \mathbb{Z}\}$. For any two
 orbit 
equivalent automorphism groups, their groups of cohomology are 
isomorphic
(see Proposition \ref{top group isomorphism} below). 
This means that, studying cocycles of $\Gamma$, it suffices to work with
 cocycles  of the group $\{T^{n} : n \in \mathbb{Z}\}$. The benefit of
this fact  is that we can explicitly write down the formula for $\Z$-
cocycles as
in  (\ref{eq z-cocycle}). Hence (as was mentioned above), every 
cocycle $a: \mathbb{Z} \times X \rightarrow G$ of 
$\{T^n, n \in \mathbb{Z}\}$ is represented by  a Borel function from 
$X$ to $G$. 
\end{remark}

In the following definition, we discuss several topologies on $\mathcal{F}(X, G)$ 
which are analogous to the topology defined by convergence of measure. 

\begin{definition}\label{def cocycle topology} The  topologies $\tau_1,
\tau_2$, $\tau_3$, and $\tau_4$ on $\mathcal{F}(X, G)$ are defined by 
their bases of neighborhoods $\mathcal{U}$, $\mathcal{U'}$, $\mathcal{W}$ 
and $\mathcal{W'}$, respectively, where
$\mathcal{U} = \{U(f; \mu_1,...,\mu_n; \epsilon,\delta)\}$, $\mathcal{U'} = 
\{U'(f; \mu_1,...,\mu_n; \epsilon)\}$, $\mathcal{W} = \{W(f; \mu_1,...,
\mu_n; \epsilon)\}$, $\mathcal{W'} = \{W'(f; \mu_1,...,\mu_n; \epsilon) \}
$, and
\begin{equation}\label{eq t_1}
    U(f; \mu_1,...,\mu_n; \epsilon,\delta)  :=  \{g \in \mathcal{F} : \mu_i ( \{x: |f(x) - g(x)| > \epsilon \} ) < \delta, \forall i = 1,..,n\}, 
\end{equation}
\begin{equation}\label{eq t_2}
   U'(f; \mu_1,...,\mu_n; \epsilon)  :=  \{g \in \mathcal{F} : \mu_i ( \{x: |f(x) - g(x)| > \epsilon \} ) < \epsilon, \forall i = 1,..,n\} ,
\end{equation}
\begin{equation}\label{eq t_3}
    W(f; \mu_1,...,\mu_n; \epsilon)  :=  \{g \in \mathcal{F} :  \int_{X}  \min \, (|f(x)-g(x)|,\, 1)\, d\mu_{i} < \epsilon, \forall i = 1, ...,n \},
\end{equation}
\begin{equation}\label{eq t_4}
    W'(f; \mu_1,...,\mu_n; \epsilon)  :=  \{g \in \mathcal{F} : \int_{X} \frac{|f(x)-g(x)|}{1 + |f(x)-g(x)|} \, d\mu_{i} < \epsilon, \forall i = 1, ...,n \}. 
\end{equation}
In the above definitions, we take  $f \in \mathcal{F}(X, G)$,  $\mu_1, ..,\mu_n \in \mathcal{M}_1(X)$,  $\epsilon, \delta > 0$, and $n \in 
\mathbb{N}$.
\end{definition}

\begin{theorem}\label{thm equivalent topologies} All the topologies 
$\tau _1$, $\tau _2$, $\tau _3$, and $\tau _4$ from Definition 
\ref{def topology} coincide on the  group $\mathcal F(X, G)$.
\end{theorem}

\begin{proof}
 For the entire proof, we assume that $i \in \{1,2,....,n\}$. 
 Also note that the notation $\tau_j \subset \tau_k$,  for topologies 
 $\tau_j, \tau_k$, $j,k \in \{1,2,3,4\}$ , $j \neq k$, means that $\tau_k$ 
is stronger than $\tau_j$. Because our topologies are determined in terms 
of the bases of neighborhoods, it suffices to check that the base for 
$\tau_k$ contains that for $\tau_j$.  For example, $\tau_1 \subset
 \tau_2$, implies that for every $f \in \mathcal{F}(X, G)$ and  a base
 element $U(f; \mu_1,...,\mu_n; \epsilon,\delta)$ of $\tau_1$ containing 
 $f$, there exists a base element $U'(f; \mu_1,...,\mu_n; \kappa)$ of 
 $\tau_2$ such that $U'(f; \mu_1,...,\mu_n; 
 \kappa)  \subset U(f; \mu_1,...,\mu_n; \epsilon,\delta)$.
\medskip

(1) \textit{$\tau_1$ coincides with $\tau_2$ on $\mathcal{F}(X,G)$}: 

Clearly, for $\delta = \epsilon $, we have $\tau_2 \subset \tau_1$.  To
 prove the converse, we will show, as mentioned above, that for a base
element $U(f;\mu_1,...,\mu_n;\epsilon,\delta) \in \mathcal{U}$, there 
exists a base element  $U'(f; \mu_1,...,\mu_n; \kappa) \in \mathcal{U'}$ 
 such that $U'(f; \mu_1,...,\mu_n; \kappa) \subset U(f;\mu_1,...,\mu_n;
 \epsilon,\delta)$. 

If $0< \epsilon < \delta $, take $\kappa = \epsilon$, and we are done, since 
for $\epsilon < \delta $, $U'(f; \mu_1,...,\mu_n; \epsilon) \subset U(f;
\mu_1,...,\mu_n;\epsilon,\delta)$.

Now assume that $0 < \delta < \epsilon$. Then take
$\kappa = \delta$ and show that $U'(f; \mu_1,...,\mu_n; \delta) \subset U(f;\mu_1,...,\mu_n;
 \epsilon,\delta)$. To see this, take any function  $g \in U'(f; \mu_1,...,\mu_n; \delta)$ and note that $0 < \delta < \epsilon $ implies
\begin{equation*}
    \{x: |f(x) - g(x)| > \epsilon \} \subset \{x: |f(x) - g(x)| > \delta \}.
\end{equation*}
Thus for all $i$, we have
\begin{equation*}
    \mu_i ( \{x: |f(x) - g(x)| > \epsilon \} ) \leq \mu_i ( \{x: |f(x) - g(x)| > \delta \} ) < \delta.
\end{equation*}
Hence  $g \in U(f;\mu_1,...,\mu_n;\epsilon,\delta)$ as needed. 
\medskip 

(2)  \textit{$\tau_1$ coincides with $\tau_3$ on $\mathcal{F}(X,G)$} : 

First we show that  $\tau_3 \subset \tau_1$.  We need to verify that,
for any neighborhood $W(f; \mu_1,...,\mu_n; \epsilon) \in \mathcal{W}$, 
there exists a neighborhood $U(f; \mu_1,...,\mu_n; \epsilon',\delta) \in 
\mathcal{U}$ such that $U(f; \mu_1,...,\mu_n; \epsilon',\delta) \subset 
W(f; \mu_1,...,\mu_n; \epsilon)$. 

To see this, let $\epsilon' = \epsilon/ 4 $, and consider $g \in U(f; \mu_1,...,
\mu_n; \epsilon',\delta)$, where $\delta > 0$ will be chosen later. Then,
 for all $i$, we have
\begin{equation}\label{eq 3a}
    \mu_i ( \{x: |f(x) - g(x)| > \epsilon/4 \} ) < \delta .
\end{equation}
We will prove that 
\begin{equation}\label{eq 3b}
    \int_{X}  \min \, (|f(x)-g(x)|,\, 1)\; d\mu_{i} < \epsilon .
\end{equation}
Choose a Borel set $B$ such that 
$$
\min \, (|f(x)-g(x)|,\, 1) = \left\{
\begin{array}{ll}
    |f(x)-g(x)|, & \mbox{$x\in B$}\\
    \\
    1, & \mbox{$x\in X\setminus B.$}
\end{array}
\right.
$$
Define $Q = \{x \in B : |f(x) - g(x)| > \epsilon/4 \} $. Then, for all $i$,
\begin{equation*}
    \int_{B} |f-g| \, d\mu_{i} = \int_{Q} |f-g| \, d\mu_{i} + \int_{B\setminus Q} |f-g|\, d\mu_{i}.
\end{equation*}
Choose $\delta > 0$, sufficiently small such that the condition $\mu_i(Q) < \delta$ implies 
$$
\int_{Q} |f-g| \, d\mu_{i} < \epsilon/4.
$$ 
For $x \in B\setminus Q$, we have $|f-g| \leq \epsilon/4 $. Since every $\mu_i$ is a probability measure, we obtain
\begin{equation}\label{eq 3c}
    \int_{B\setminus Q} |f-g| \, d\mu_{i} < (\epsilon/4) \, \mu_i (B\setminus Q) < \epsilon/4 .
\end{equation}
Thus, for all $i$, we see that
\begin{equation}\label{eq 3c'}
    \int_{B}  \min \, (|f(x)-g(x)|,\, 1)\, d\mu_{i} < \epsilon/2 .
\end{equation}
Using $(\ref{eq 3a})$ and choosing $\epsilon/4 < 1$ and $\delta < \epsilon/2$ we get  $\mu_i (X \setminus B) < \epsilon/2 $.
Therefore, for all $i$, the following inequality holds
\begin{equation}\label{eq 3c''}
    \int_{X\setminus B}  \min \, (|f(x)-g(x)|,\, 1)\, d\mu_{i} < \epsilon/2 .
\end{equation}
Relations  $(\ref{eq 3c'})$ and $(\ref{eq 3c''})$ imply  $(\ref{eq 3b})$. 
This completes the proof of $\tau_3 \subset \tau_1$.  

Now we prove that $\tau_1 \subset \tau_3$.  We show that, for a base 
element $U(f; \mu_1,...,\mu_n; \epsilon,\delta)$ \newline 
$ \in \mathcal{U}$, there exists a base element $W(f; \mu_1,...,\mu_n; 
\kappa) \in \mathcal{W}$ such that $W(f; \mu_1,...,\mu_n; \kappa) 
\subset U(f; \mu_1,...,\mu_n; \epsilon,\delta)$. 

For this, let $\kappa = \epsilon \delta $ and let $g \in W(f; \mu_1,...,\mu_n; 
\kappa) $. Then for all $i$, we get 
\begin{equation}\label{eq 3d}
    \int_{X}  \min \, (|f-g|, \, 1) \, d\mu_{i} < \epsilon \delta .
\end{equation}
 Assume, toward a contradiction, that $g \notin U(f; \mu_1,...,\mu_n;
  \epsilon,\delta)$, i.e., 
\begin{equation}\label{eq 3e}
    \mu_i(\{x : |f(x)- g(x)| > \epsilon \}) \geq \delta. 
\end{equation}
 Denote  $P = \{x : |f(x)- g(x)|> \epsilon \}$; then for all $i$,
\begin{equation*}
    \int_{X} |f-g| \, d\mu_{i} \geq \int_{P} |f-g| \, d\mu_{i} \geq \epsilon 
    \delta.  
\end{equation*}
which contradicts $(\ref{eq 3d})$. Hence, we conclude that  $g \in U(f; 
\mu_1,...,\mu_n; \epsilon,\delta)$ as needed. 
\medskip

(3) \textit{$\tau_1$ coincides with $\tau_4$ on $\mathcal{F}(X,G)$}: 

To see this, let $K = \{x: |f(x) - g(x)| > \epsilon \}$ and note  that the 
equality
\begin{equation}\label{eq 3f}
     \{x: |f(x) - g(x)| > \epsilon \} =  \{x: \frac{|f(x) - g(x)|}{1 + |f(x) - g(x)|} > \frac{ \epsilon}{1 + \epsilon} \} : = K
\end{equation}
holds.  We first show that $\tau_1 \subset \tau_4$. Let  
$U(f; \mu_1,...,\mu_n; \epsilon,\delta) \in \mathcal{U}$ be a neighborhood
from $\tau_1$. Show that  there exists a neighborhood
 $W'(f; \mu_1,...,\mu_n; \kappa) \in \mathcal{W'}$  such that 
 $W'(f; \mu_1,...,\mu_n; \kappa) \subset U(f; \mu_1,...,\mu_n; \epsilon,
 \delta)$. 

Let $\kappa = \dfrac{\epsilon \delta}{1+ \epsilon}$ and let $g \in W'(f; 
\mu_1,...,\mu_n; \kappa)$, then 
\begin{equation*}
   \int_{X} \frac{|f(x)-g(x)|}{1 + |f(x)-g(x)|} d\mu_{i} < \frac{\epsilon \delta}
   {1+ \epsilon}, \quad \forall i.
\end{equation*}
Relation $(\ref{eq 3f})$ implies 
\begin{equation*}
    \frac{ \epsilon}{1 + \epsilon} \chi_{K} < \frac{|f(x)-g(x)|}{1 + |f(x)-g(x)|} 
    \chi_{K} < \frac{|f(x)-g(x)|}{1 + |f(x)-g(x)|}   
\end{equation*}
Hence,
\begin{equation*}
    \mu_i(K) < \frac{1+ \epsilon}{\epsilon} \int_{X}\frac{|f(x)-g(x)|}{1 + |
    f(x)-g(x)|} \,d\mu_i < \delta.
\end{equation*}
which implies that $g \in U(f; \mu_1,...,\mu_n; \epsilon,\delta)$. 
\medskip

It remains to prove that $\tau_4 \subset \tau_1$.  Show that
for a neighborhood $W'(f; \mu_1,...,\mu_n; \epsilon) \in \mathcal{W}$,
 there exists a basis element $U(f; \mu_1,...,\mu_n; \epsilon,\delta) \in 
 \mathcal{U}$ such that $U(f; \mu_1,...,\mu_n; \epsilon,\delta) 
\subset W'(f; \mu_1,...,\mu_n; \epsilon)$. 
Take a function $g \in U(f; \mu_1,...,\mu_n; \epsilon,\delta)$, then 
\begin{equation*}
\mu_i(K) = \mu_i(\{x: |f(x) - g(x)| > \epsilon \}) < \delta.
\end{equation*}
 Choose $\delta$ such that $\delta < \dfrac{\epsilon^2}{1+\epsilon}$, then 
 we obtain for each measure $\mu_i$ 
\begin{equation*} 
\begin{split}
\int_{X}\frac{|f(x)-g(x)|}{1 + |f(x)-g(x)|} \, d\mu_i  & \leq \int_{K}\frac{|f(x)-g(x)|}{1 + |f(x)-g(x)|} \, d\mu_i + \int_{K^c}\frac{|f(x)-g(x)|}{1 + |f(x)-g(x)|} \, d\mu_i  \\
 & \leq \mu_i(K) + \frac{\epsilon}{1+ \epsilon} \mu_{i} (K^c)  \\
 & < \delta + \frac{\epsilon}{1+ \epsilon} < \epsilon.
\end{split}
\end{equation*}

Thus, $g \in W'(f; \mu_1,...,\mu_n; \epsilon)$ as needed. 
\end{proof}

\begin{remark}\label{equiv top}Since the topologies $\tau _1$, $\tau _2$, 
$\tau _3$ and $\tau _4$, on $\mathcal{F}(X, G)$ coincide, we will will use 
the notation $\mathcal{T}$ to denote them.
\end{remark}
 
\begin{theorem}\label{thm sep top group} $\mathcal{F}(X, G)$ is separable 
Hausdorff topological group with respect to the topology $\mathcal{T}$.
\end{theorem}

\begin{proof} We denote by $\mathcal{A} = \{B_i\}_{i \in \mathbb{N}}$
the countable base for the space $X$ which generates $\B$. Recall that
$G$ is an abelian l.c.s.c group with identity $0$. Let $G_0$ be a countable
dense subgroup of $G$. Denote by $\alpha_i \chi_{B_i}$ a 
function  $X\to G $ which takes the value $\alpha_i \in G$ 
on the set $B_i$ and is $0$ everywhere else. Note that we refrain from 
using the term ``characteristic function'' as $G$ is an additive abelian group 
with identity $0$ but the notion of multiplicative identity is not defined. 

Consider the set $\mathcal{S}(X, G_0)$ of all finite linear combinations of
 such constant functions with values in $G_0$, i.e., they can be 
 described as piecewise constant functions that take values from $G_0$ 
on sets from the family $\mathcal A$ and are zero everywhere
else. We will call elements of $\mathcal{S}(X, G_0)$ simple functions.

For notational purpose, we will denote such a function as follows
\begin{equation*}
 f(x)= \sum_{l=1}^p  \alpha_l \chi_{B_l} (x)
\end{equation*}
where $\alpha_l \in G_0$ and $B_l \in \mathcal{A}$ for $l = 1,2,...,p$.

We first observe that the set  $\mathcal{S}(X, G_0)$ is a 
countable subset of $\mathcal{F}(X, G)$. In what follows we will show 
that $\mathcal{S}(X, G_0)$ is dense in $\mathcal{F}(X,G)$ with respect to
the topology $\mathcal{T}$. 

For $f \in \mathcal{F}(X, G)$, consider a neighborhood of $f$ 
\begin{equation*}
    U(f; \mu_1,...,\mu_n; \epsilon,\delta)  =  \{g \in \mathcal{F} : \mu_i 
    (\{x: |f(x) - g(x)| > \epsilon \} ) < \delta, \forall i = 1,2,..,n\} 
\end{equation*}
where $\mu_1, ... , \mu_n \in \mathcal{M}_1(X)$. To prove the result, it
suffices to find an element from the set $\mathcal{S}(X, G_0)$ in  $U(f; 
\mu_1,...,\mu_n; \epsilon,\delta)$. 

Since $f \in \mathcal{F}(X, A)$ is a Borel function, there exists a sequence $
\{s_j\}_{j \in \mathbb{N}}$ of simple function taking value in $G_0$ which 
converges pointwise to $f$. Again using the same notation as above we 
denote $s_j$ as follows 
\begin{equation*}
    s_j = \sum_{k=1}^{m} \alpha_{k,j} \chi_{E_{k,j}}, \quad  j \in 
   \mathbb{N}.
\end{equation*}
where $\alpha_{k,j} \in G_0$, $ \forall k = 1,2,...,m$ and $E_{k,j} = 
\{x \in X : s_j(x)= c_{k,j}, \forall k\}$.

For the measure $\mu_1$, we use Egoroff's theorem and find a Borel set
 $F_1 \in X$ such that $s_j \rightarrow f$ uniformly on $F_1$, and 
$\mu_1(X \setminus F_1) < \dfrac{\delta}{n}$. (Note that this convergence 
is uniform  in the usual sense: for $\epsilon > 0$ there exists $N_1 \in 
\mathbb{N}$, such that for all $ j > N_1$ and for all $x \in F$, 
$|f(x)- s_j(x)| < \epsilon$). Similarly, there exists a Borel set $F_2 \subset 
F_1$ such that the sequence $(s_j)$ converges  uniformly to $f$ on 
$F_2$, and $\mu_2(F_1 \setminus F_2) < \dfrac{\delta}{n}$. 
Repeating this process $n$ times we obtain a Borel set $F \subset X$ 
such that the convergence  $s_j \rightarrow  f$ is uniform on $F$, and, for 
$i = 1,2,..,n$, we have $\mu_i(X \setminus F) <\delta$. Hence for any 
$\epsilon > 0 $ one can find some  $N \in \mathbb{N}$ such that 
$|f(x)- s_t(x)| < \epsilon$ for $t > N$ and $x \in F$. 
 In other words, $\mu_i ( \{x: |f(x) - s_t(x)| > \epsilon \} ) < \delta\}$, 
 $i = 1, 2, ...,n$.

This implies that, for $t> N$, the functions $s_t = \sum_{k=1}^{m} 
\alpha_{k,t} \chi_{E_{k,t}} $ belong to $U(f; \mu_1,...,\mu_n; \epsilon,
\delta) $.  Since this is true for any $\delta > 0$, choose $N$ such that 
for $t \in N$ 
we have $s_t \in U(f; \mu_1,...,\mu_n; \epsilon,\dfrac{\delta}{q})$, where $q$ is a positive integer to be chosen later. It follows that
\begin{equation*}
    \mu_i ( \{x: |f(x) - s_t(x)| > \epsilon \} ) < \frac{\delta}{q}, \quad 
    i = 1, 2, ...,n.
\end{equation*}
In other words, we obtain that, for $k =1,2,..,m$,
\begin{equation}\label{eq 3g}
    \mu_i ( \{x \in E_{k,t}: |f(x) - \alpha_{k,t}| > \epsilon \} ) < \frac{\delta}{q}, \quad i = 1, 2, ...,n.
\end{equation}
where $\alpha_{k,t} \in G_0$, $ k = 1,2,...,m$ and $ t > N$. 


Since each $E_{k,t}$ is a Borel set,  it can approximated by an open set, 
i.e.,  there exists an open set $O_{k,t}^1, ... , O_{k,t}^n$ such that 
$$
\mu_1 (O_{k,t}^i\ \Delta \  E_{k,t}) < \frac{\delta}{2q}, \quad i = 1, .., n.
$$ 
Define $O_{k,t} = \bigcap_{i=1}^n O_{k,t}^i$, then, for every $i= 1,2,...n$,
one has
$$
\mu_i (O_{k,t}\  \Delta \ E_{k,t}) < \frac{\delta}{2q}.
$$
 Each open set $O_{k,t}$ is a countable union of base elements i.e. $O_{k,t}$ = $\bigcup_{i \in \mathbb{N}} B_i$, where $B_i \in \mathcal{A}$. Thus there exists a finite number, $r (k,t) \in \mathbb{N}$ such that for every $i= 1,2,...n$, 
\begin{equation}\label{eq p10}
\mu_i \Big(\big(\bigcup_{l=1}^{r(k,t)} B_l\big)\  \Delta \ 
 O_{k,t}\Big) < \frac{\delta}{2q}.
\end{equation}
Let us denote by $I_{k,t}$ the index set $I_{k,t} = \{1,2,..,r(k,t)\}$. 
Thus, \eqref{eq p10}   implies that 
$$
\mu_i \Big(\big(\bigcup_{l \in I_{k,t}} B_l\big) \ \Delta \ E_{k,t}\Big) < \frac{\delta}{q}.
$$ 
Since 
$$\Big\{x \in \Big(\bigcup_{l \in I_{k,t}} B_l\Big)\  \Delta \ E_{k,t}  : |f(x) - \alpha_{k,t}| > \epsilon \Big\} \subset \Big(\bigcup_{l \in I_{k,t}} B_l\Big) \ \ \Delta \ E_{k,t},
$$ we have
\begin{equation}\label{eq 3h}
    \mu_i \Big( \Big\{x \in \Big(\bigcup_{l \in I_{k,t}} B_l\Big)\  \Delta \ 
 E_{k,t}: |f(x) - \alpha_{k,t}| > \epsilon \Big\} \Big) < \frac{\delta}{q}, \ \ 
  i = 1, 2, ...,n.
\end{equation}
Now take $q = 2m$ where $m$ is as in the definition of $s_j$ above, 
then by $(\ref{eq 3g})$ and $(\ref{eq 3h})$, we obtain
\begin{equation}\label{eq 3i}
    \mu_i \Big( \Big\{x \Big(\in \bigcup_{l \in I_{k,t}} B_l\Big) : |f(x) - \alpha_{k,t}| > \epsilon \Big\} \Big) < \frac{\delta}{m}, \, \forall i = 1, 2, ...,n.
\end{equation}
Note that $(\ref{eq 3i})$ is true for all $m = 1,2,...,k$. Let $I_{t}= \bigcup_{k=1}^m I_{k,t}$, then 
\begin{equation}\label{eq 3i_1}
    \mu_i \Big( \Big\{x \in \Big( \bigcup_{l \in I_{k}} B_l \Big) : |f(x) - \alpha_{k,t}| > \epsilon \Big\} \Big) < \delta, \ \ \  i = 1, 2, ...,n.
\end{equation}
 Define the sequence of functions $s'_t$, for $t \in \mathbb{N}$, as follows
$$
s'_t(x)= \left\{
\begin{array}{ll}
    \alpha_{k,t}, & \mbox{$x\in B_l,\, l \in I_{k,t}$}\\
    \\
    0, & \mbox{$x\notin B_l,\, l \in I_{k,t}.$}
\end{array}
\right.
$$
Then, by $(\ref{eq 3i_1})$, we have
\begin{equation}\label{eq 3j}
    \mu_i ( \{x : |f(x) - s'_t| > \epsilon \} ) < \delta, \, \forall i = 1, 2, ...,n.
\end{equation}
Relation (\ref{eq 3j}) implies that   $s'_t \in U(f; \mu_1,...,\mu_n; \epsilon,\delta) $ for $t> N$. Therefore $\mathcal{S}(X, G_0)$ is dense in 
$\mathcal{F}(X, G)$, and $\mathcal{F}(X, G)$ is a separable space.
\medskip

To prove the second part of the theorem, we will show that $\mathcal{F}(X, G)$ is a topological group with respect to the  topology $\mathcal{T}$. We will do it for the topology $\tau_3$ (see Definition \ref{def cocycle topology}) because it is easier to work this topology.  Note the following
 facts: 

(i) $W(f; \mu_1,...,\mu_n; \epsilon) = -W(-f; \mu_1,...,\mu_n; \epsilon)$.

(ii) $W(f; \mu_1,...,\mu_n; \epsilon/2) + W(g; \mu_1,...,\mu_n; \epsilon/2) \subset W(f+g; \mu_1,...,\mu_n; \epsilon) $

Both (i) and (ii) are clear by the definition of $W(f; \mu_1,...,\mu_n; \epsilon)$ and $W(g; \mu_1,...,\mu_n; \epsilon)$. It follows from (i) 
that the map $f\mapsto -f$ is continuous and (ii) implies that the map $ (f,g) \mapsto f+g $ is also continuous.

To see that $\mathcal{F}(X, G)$ is Hausdorff in the topology $\mathcal{T}$, 
consider $f,g \in  \mathcal{F}(X, G)$ such that $f \neq g$. Then there exists 
 $x \in X$,such that $f(x) \neq g(x)$. We work with topology $\tau_1$ and put $\mu_1 = \delta_x$ (the Dirac measure at $x$). Note that,  
 for $\delta < 1$, the open set $U$ (defined below) contains $f$ but does not contain $g$:
$$
U = \{h \in \mathcal{F} : \delta_x ( \{y: |f(y) - h(y)| > \epsilon \} ) < \delta \} 
$$
For $\delta < 1$, we get 
$$
U = \{h \in \mathcal{F} : \delta_x ( \{y: |f(y) - h(y)| > \epsilon \} ) =0 \} 
$$
Therefore $x \notin \{y: |f(y) - h(y)| > \epsilon \}$ and  $g \notin U$. 
\end{proof}

\begin{proposition}\label{closed cocycles} Let $\Gamma$ be a hyperfinite countable subgroup of $Aut(X,\mathcal{B})$. The group $Z^1(\Gamma \times X,G) $ is closed in $\mathcal{F}(\Gamma \times X, G)$, and  it is a separable topological group. 
\end{proposition}

To prove Proposition \ref{closed cocycles}, we will show that if  
$\{a_n\} \subset Z^1(\Gamma \times X, G)$ is a sequence
of cocycles such that $a_n \rightarrow a $ in $\tau_1$, then   $a \in  
Z^1(\Gamma \times X, G) $.  For this,  we will prove the following lemma. 


\begin{lemma}\label{equality} Let  $\{a_n\}$ be a sequence of cocycles
 from  $Z^1(\Gamma \times X, G)$. Then $a_n \rightarrow a $ in the 
 topology $\tau_1$ 
 if and only if for every  $x \in X$ there exists $n(x) \in \mathbb{N}$ 
 such that  $a_n(x) = a(x)$ for all $n > n(x)$.  
\end{lemma}

\begin{proof}
 As mentioned in  Remark $\ref{R_a}$ the group $\Gamma$ is 
orbit equivalent to a group generated  by a single automorphism $\{T^n : 
 n \in \mathbb{Z}\}$. It gives us the possibility to represent cocycles 
$a_n$ as functions on $X$ with values in the group $G$. 

Assume now that $a_n \rightarrow a $ in $\tau_1$.  Then, for every 
$\epsilon,\delta > 0$ there exists $n(x) \in \mathbb{N}$ such that 
$a_n \in U(a; \mu_1,...,\mu_p; \epsilon,\delta)$ for $n > n(x)$ (here  
$\mu_1, ..,\mu_p \in \mathcal{M}_1(X)$ as usual). Fix $x \in X$ and take 
$\mu_1 = \delta_x$ (the Dirac measure at $x$). Thus we have 
$\delta_x(\{y: |a_n (y)- a (y)|> \epsilon\}) < \delta$. For  $\delta < 1$ we
get $\delta_x(\{y: |a_n (y)- a (y)|> \epsilon\}) =0$. Hence $x \notin \{y: |
a_n (y)- a (y)|> \epsilon\}$ for all $n > n(x)$. We conclude that
 $a_n(x) = a(x)$. 

Conversely, suppose that, for every $x \in X$, there exists $n(x) \in 
\mathbb{N}$ such that  $a_n(x) = a(x)$ for all $n > n(x)$.  Define 
$X_n = \{x \in X : a_m(x) = a(x), \forall m \geq n \}$, $ n \in \mathbb N$. 
Note that $X_n \subset X_{n+1}$, and $\bigcup_{n=1}^{\infty} X_n = X$. 
For every $\mu \in \mathcal{M}_1(X)$, we see that $\mu(X_n) \rightarrow 
1$ as $n \rightarrow \infty$. Take a neighborhood $U(a; \mu_1,...,
\mu_p; \epsilon,\delta)$ and find  $n_0 \in \mathbb{N}$ such that 
$\mu_i (X_n) > 1- \delta$ for  $n > n_0$, $i=1,2,...,p$. Note that, for all 
$n \in \mathbb{N}$, 
$$
\{x \in X: |a_n(x)-a(x)| > \epsilon\} \subset X \setminus X_n. 
$$
 Thus $\mu_i(\{x \in X: |a_n(x)-a(x)| > \epsilon\}) < \mu_i(X 
\setminus X_n) < \delta$. Hence, for $n > n_0$, 
we deduce that $\mu_i(\{x \in X: |a_n(x)-a(x)| > \epsilon\}) < \delta$ 
as needed.  
\end{proof}
\medskip


\noindent\textit{Proof of Proposition 3.4}. We switch back to considering $a_n$ and $a$ as functions from $\Gamma \times X$ to $G$. Since $a_n \in  Z^1(\Gamma \times X,G) $, $  a_n(\gamma_1 \gamma_2,x) = a_n (\gamma_1, \gamma_2 x) + a_n (\gamma_2, x)$,  $ \forall \gamma_1, \gamma_2 \in \Gamma$. 

For a fixed $x \in X$, let $n_0 = max \{n(x), n(\gamma_2 x)\}$, then for $n > n_0$, we have 
$$
a_n(\gamma_1 \gamma_2,x) = a(\gamma_1 \gamma_2,x),
$$
$$
a_n(\gamma_1, \gamma_2 x) = a(\gamma_1, \gamma_2 x),
$$
$$
a_n (\gamma_2, x) = a(\gamma_2, x).
$$

\noindent Hence $a(\gamma_1 \gamma_2,x) = a(\gamma_1, \gamma_2 x) + a(\gamma_2, x)$,  $ \forall \gamma_1, \gamma_2 \in \Gamma$. Since we can do this for every $x \in X$, $a \in Z^1(\Gamma \times X,G) $. 
\hfill{$\square$}


\begin{proposition}\label{top group isomorphism} Let $\Gamma_i \in  
Aut(X_i,\mathcal{B}_i)$, $i = 1,2, $ be two orbit equivalent countable Borel
 automorphism groups. Then there exists a topological group isomorphism 
 $ \widetilde{\varphi} : Z^1(\Gamma_1 \times X_1,A) \rightarrow$ 
 $Z^1(\Gamma_2 \times X_2,A)$ which carries coboundaries to 
 coboundaries. 

\end{proposition}

\noindent\textit{Proof}. Since $\Gamma_1$ and $\Gamma_2$ are orbit equivalent, there exists a Borel map $\varphi : X_1 \rightarrow X_2$, such
that $\varphi [\Gamma_1] = [\Gamma_2] \varphi$. Define  
$\widetilde{\varphi} : Z^1(\Gamma_1 \times X_1,A) \rightarrow$ $Z^1(\Gamma_2 \times X_2,A)$ as 
\begin{equation*}
    \widetilde{\varphi} \circ a_1 (\gamma_2,x_2) = a_1 (\varphi^{-1} \gamma_2 \varphi, \varphi^{-1} x_2 )
\end{equation*}
for $a_1 \in Z^1(\Gamma_1 \times X_1, G)$ and $(\gamma_2,x_2) \in \Gamma_2 \times X_2$. Then, $\widetilde{\varphi} $ is an isomorphism by definition. If $a_1$ is a coboundary, $a_1(\gamma_1, x_1) = c(\gamma_1 x_1)- c(x_1)$, where $c : X \rightarrow G$ is a Borel map.
\begin{equation*}
\widetilde{\varphi} \circ a_1 (\gamma_2,x_2) = a_1 (\varphi^{-1} \gamma_2 \varphi, \varphi^{-1} x_2 ) = c(\varphi^{-1} \gamma_2 \varphi (\varphi^{-1} x_2)) - c(\varphi^{-1} x_2)
\end{equation*}
is also a coboundary. \hfill{$\square$}

\vspace{3mm} 

\begin{corollary}\label{invariance}  For a Borel automorphism group $\Gamma$ of $(X,\mathcal{B})$ the first cohomology group $H^1(\Gamma \times X, G) = Z^1(\Gamma \times X,G) /
 B^1(\Gamma \times X, G)$ is an invariant of orbit equivalence.

\end{corollary}

\begin{remark}\label{coboundary structure}   In general, 
$B^1(\Gamma \times X, G)$ is not closed in the topology described above. Hence 
$H^1(\Gamma \times X, G)$ should be considered as  an abstract
 group that does not inherit the topological or Borel structure.
\end{remark}

\begin{remark}\label{rem ctbl}
Let $Ctbl(X)$ be defined as the subset of $Aut(X, \B)$ consisting of all
automorphisms with countable support, that is 
$$T\in Ctbl(X) \ 
\Longleftrightarrow  \ E(S, \mathbbm I)\ 
{\rm is \ at\ most\ countable}.$$ 
One can show that $Ctbl(X)$ is a normal subgroup which is 
closed with respect to the  uniform topology, see \eqref{eq top t} in 
Definition \ref{def topology}.
 Therefore
$\widehat{Aut}(X,{\B})= Aut(X, \B)/Ctbl(X)$ is a simple Hausdorff 
topological group
with respect to the quotient topology 
\cite{BezuglyiDooleyKwiatkowski_2006}. Considering elements from
$\widehat{Aut}(X,{\B})$, we identify Borel automorphisms which differ on at most a countable set.  Topological properties of the
group $\widehat{Aut}(X,{\B})$ are studied in 
\cite{BezuglyiMedynets_2004}. It was shown that the quotient
topology on $\widehat{Aut}(X,{\B})$ is in fact generated by 
neighborhoods $V(T; \mu_1, ... ,\mu_n; \varepsilon)$ where the 
measures $\mu_1, ... ,\mu_n$ are taken from 
$M^c_1(X)$, the set of all non-atomic Borel probability measures
on a standard Borel space $(X, \B)$. 

Using a similar approach, we identify two functions $f$ and $g$ if
they differ on at most a countable set. In other words, we define the
quotient set $\widehat{\mc F}$ with elements $\widehat g = 
\{g\circ T : T\in Ctbl(X)\}$ where $g \in \mc F(X, \B)$. Then one can 
show that the quotient topology $\widehat {\tau}$ on $\widehat
{\mc F}$ is
 defined by neighborhoods $V(f; \mu_1, ... ,\mu_k; \epsilon)$ where
 the measures $\mu_1, ... ,\mu_k \in M^c_1(X)$. 
\end{remark}

Based on Remark \ref{rem ctbl}, we can obtain the following result. 
The proof is left for the reader because we do not use this result in the
paper.

\begin{proposition}
Let $\widehat \tau$ be the topology on $\widehat{\mc F}
(X, G)$ defined as in 
Remark \ref{rem ctbl} by atomless  measures  from 
$M_1^c(X)$. Then, for $\widehat f_n $ and $\widehat f$ from 
$\widehat {\mc F}$,  
$$
\widehat f_n \ \stackrel{\widehat{\tau}}\longrightarrow \ \widehat f
$$
if and only  $(\widehat f_n)$ converges to $\widehat f$ uniformly.

\end{proposition}

\section{Density of coboundries for hyperfinite Borel actions} \label{Sec Density}
In this section we prove following result.

\begin{theorem}\label{thm dense} Let $\Gamma
 \subset Aut(X,\mathcal{B})$ be a hyperfinite Borel automorphism group. 
 Then $B^1(\Gamma \times X, G)$ is dense in $Z^1(\Gamma \times X, G)$
 with respect to the topology $\mathcal T$ where $G$ is a l.c.s.c. group.

\end{theorem}

Since $\Gamma$ is hyperfinite, it is orbit equivalent to a Borel 
$\mathbb{Z}$-action. By Corollary $\ref{invariance}$, the first cohomology 
group is an invariant of orbit equivalence. Hence, without loss of generality, 
it suffices to prove the statement for a single Borel automorphism $T \in
 Aut(X,\mathcal{B})$. 
To prove the theorem, we will use the Kakutani tower construction for an
 aperiodic Borel automorphism which gives the possibility to use periodic
automorphisms to approximate $T$. This construction is described in 
 \cite[Chapter 7]{Nadkarni_2013} and \cite{BezuglyiDooleyKwiatkowski_2006}. We  include it here for convenience of the reader.

Recall that  a Borel set $A\subset X$ is called a {\it complete section} (or 
simply a {\it $T$-section}) for an automorphism $T \in Aut(X,\mathcal{B})$ 
if every $T$-orbit meets $A$ at least once. If there exists a complete Borel 
section $A$ such that $A$ meets every $T$-orbit exactly once, then $T$ is 
called {\it smooth}. In this case, $X = \bigcup_{i\in \Z} T^iA$ and all the sets 
$T^iA$ are disjoint.
A measurable set $W$ is said to be \textit{wandering} with respect to 
$T \in Aut(X, \mathcal{B})$ if the sets $T^nW,\ n\in \mathbb{Z}$, are 
pairwise disjoint. 
The $\sigma$-ideal generated by all $T$-wandering sets in $\B$ is denoted 
by $\mathcal{W}(T)$. By the Poincar\'{e} recurrence lemma, one can state that 
given $T \in Aut(X,\mathcal{B})$ and $A\in \mathcal{B}$ there exists $N\in 
\mathcal{W}(T)$ such that for each $x\in A\setminus N$ the points $T^nx$ 
return to $A$ for infinitely many positive $n$ and also for infinitely many 
negative $n$. The points from
the set $A\setminus N$ are called {\it recurrent}.

\begin{remark}\label{Towers} Assume that all points from a given set $A$ 
are recurrent for a Borel automorphism $T$. Then for $x\in A$, let $n(x) = 
n_A(x)$ be the smallest positive integer such that $T^{n(x)}x \in A$ and 
$T^ix \notin A,\ 0< i < n(x)$. Let $C_k = \{x\in A\  \vert\  n_A(x) =k\},\ k
\in \N $, then $ T^kC_k \subset A $ and $ \{T^iC_k\  \vert\ i=0,...,k-1\} $ 
are pairwise disjoint. Note that some $C_k$'s may be empty. Since $T^nx\in 
A$ for infinitely many positive and negative $n$, we obtain
$$
 \bigcup_{n\ge0}T^nA =  \bigcup_{n\in \Z}T^nA = X
 $$
and
$$
X =  \bigcup_{n\ge0}T^nA  = \bigcup_{k=1}^\infty\bigcup _{i=0}^{k-1} 
T^iC_k.
$$

This union decomposes $X$ into $T$-towers $\xi_k = \{T^iC_k \ \vert \ 
i=0,..., k-1\},\ k \in \N$, where $C_k$ is the base and $T^{k-1}C_k$ is the 
top of $\xi_k$. Depending on $T$,  the set of these towers $\{\xi_k\}$ can
 be, in general,  countable.
\end{remark}

\begin{lemma}\label{lemma markers} Let $T \in Aut(X,\mathcal{B})$ be an 
aperiodic Borel automorphism of a standard Borel space $(X,\mathcal{B})$. 
Then there exists a sequence $(A_n)$ of Borel sets such that

(i) $X=A_0 \supset A_1 \supset A_2\supset \cdots,$

(ii) $\bigcap_n A_n =\emptyset,$

(iii) $A_n$ and $X\setminus A_n$ are complete $T$-sections, $n\in \N$,

(iv)  every point in $A_n$ is recurrent, $n\in \N$. 
\end{lemma}

\begin{proof}
See \cite[Lemma 4.5.3]{BeckerKechris_1996} where (i) - (iii) have been proved in more general 
settings of countable Borel equivalence relations. It is shown in \cite[Chapter 7]{Nadkarni_2013}
that one can refine the choice of $(A_n)$ to get (iv). 
\end{proof} 

\begin{definition}\label{vanishing seq} A sequence of Borel sets satisfying 
conditions (i) - (vi) of Lemma $\ref{lemma markers}$ is called a 
\textit{vanishing sequence of markers}.
\end{definition}

\begin{proposition}\label{lemma periodic approx} Let $T \in Aut(X,
\mathcal{B})$ be an aperiodic Borel automorphism of a standard Borel space 
$(X, \mathcal{B})$. Then there exists a sequence of periodic automorphisms 
$(P_n)$ of $(X,\mathcal{B})$ converging to $T$ in the uniform topology (see 
Definition $\ref{def topology}$).  Moreover, the periodic automorphisms 
$P_n$ can all be taken from $[T]$.
\end{proposition}

\begin{proof} This propositions was proved in \cite[section 2]{BezuglyiDooleyKwiatkowski_2006}. 
We give the proof here as it will be used  in Lemma 
$\ref{periodic aut}$. 

Let $(A_n)$ be a vanishing sequence of markers for $T$. Then, as 
we have seen above, $A_n$ generates a decomposition of $X$ into 
$T$-towers $\xi_k(n) = \{T^iC_k(n) \ \vert\ i=0,...,k-1\}$ and $\bigcup_k 
C_k(n) = A_n$. Define
\begin{equation}\label{eq 4.1}
P_nx =\left\{ \begin{array}{ll}
Tx, & {\rm if}\  x\notin B_n= \bigcup_{k=1}^\infty T^{k-1}C_k(n)\\
\\
T^{-k+1}x, & {\rm if}\  x\in T^{k-1}C_k(n),\ {\rm for\ some}\ k
\end{array} \right.
\end{equation}
Then $P_n$ belongs to $[T]$, and the period of $P_n$ on $\xi_k(n)$ is $k$.  
Note that $P_n$ equals $T$ everywhere on $X$ except the set 
$B_n$ which is the union of the tops of the towers.

It follows from Lemma $\ref{lemma markers}$ that $(A_n)$ is a decreasing 
sequence of Borel subsets such that $\bigcap_n A_n =\emptyset$. This 
means that for any $x\in X$ there exists $n(x)$ such that $x\notin A_{n},\ 
n \ge n(x)$. Moreover, if for some $x\in X$, $P_nx = Tx$, then $P_{n+k}x 
= Tx$ for all $k$. These facts prove that, for every $x$, the sequence  
$(P_n x)$ is  eventually stabilized and it is and equal to $Tx$. Hence,
 $P_n$ converges to $T$ in the topology $\tau$. 

\end{proof}

Lemma $\ref{periodic aut}$ is  well known   in the theory of dynamical systems. We include it here for convenience of the reader.

\begin{lemma}[folklore] \label{periodic aut}  (1) Let $P$ be a periodic 
automorphism of a standard Borel space $(X, \mathcal B)$. Then any 
cocycle of $P$ is a coboundary. 

(2) The same result holds for a smooth automorphism of 
a standard Borel space $(X, \mathcal B)$. 
\end{lemma}

\begin{proof}
(1) Let  $a \in Z^1(P \times X, G)$, be a cocycle for $P$ taking value in l.c.s.c. 
abelian group $G$ with identity $0$. Denote by $C_k$ the base of 
$P$-tower $\xi_k$ where $P$ has period $k$. Then $X$ is the disjoint union
of $\xi_k$. 
We define a Borel function $f : X \rightarrow G$ by setting $f(x) = f_k(x), x
 \in \xi_k , k \in \mathbb{N}$,   where 
$$
f_{k} (x) =\left\{ \begin{array}{ll}
a(P^j,P^{-j}x) , & {\rm if}\  x\in P^j C_k, \, \, {\rm for}\ 1 \leq j \leq k-1 \\
\\
0, & {\rm if}\  x\in C_k
\end{array} \right.
$$
It suffices to check that $a$ is a coboundary on every tower $\xi_k$. 
For every $x\in X$, there exist $k$ and $j\in \{0, ... ,k-1\}$ such that 
$x \in P^jC_k$.
Let $n \in \mathbb N$, then $P^nx \in P^mC_k$ where $n = m - j + ik$.
Therefore, we have
\begin{equation*}
    f_{k} (P^n x) - f_{k} (x) = a(P^m, P^{-j} x) - a(P^j, P^{-j} x). 
\end{equation*}
 Since, $0=a(P^j P^{-j},x)= a(P^j,P^{-j}x) + a(P^{-j}, x)$, we obtain 
\begin{equation*}
    f_{k} (P^n x) - f_{k} (x) = a(P^m, P^{-j} x) + a(P^{-j}, x) 
    = a(P^m P^{-j}, x) = a(P^n, x).
\end{equation*}
Hence, $a$ is a coboundary. 

Statement (2) is proved analogously.
\end{proof} 




Let $T \in Aut(X, \B)$ and $f$  be a Borel function on $X$. By $a(f)$ we
denote the cocycle generated by $f$:
\begin{equation}\label{ eq a(f)}
a(f)(j, x) =  
\begin{cases} 
 f(x)+ f(Tx)+...+f(T^{j-1}x), & \quad j\geq 1 \\ 
0, & \quad j = 0\\ 
-f(T^{-1}x) - f(T^{-2}x) - ... -f(T^{j}x), & \quad j\leq-1
\end{cases}
\end{equation}

\begin{lemma}\label{lem conv}
Suppose a sequence of Borel functions $(f_i)$ converges to $f$ in the 
topology $\mc T$. Then the sequence of cocycles $a(f_i)$ converges to
$a(f)$, i.e.,  for every $j\in \Z$,
 $$
 a(f_i)(j, x) \stackrel{\mc T} \longrightarrow  a(f)(j, x), \quad i  \to \infty.
$$
\end{lemma}

\begin{proof}
To prove the lemma, we need to show that for any positive $\epsilon$ and 
$\delta$ and for any finite set of Borel probability measures $\mu_1, ... ,
\mu_n$ there exists $N\in N$ such that 
\be\label{eq nts}
\mu_l(\{x \ : \ | a_i(j, x) - a(j,x) | > \epsilon \}) <\delta, \quad l = 1,..., n.
\ee

Fix a natural number $j$ (the case of negative $j$ is considered similarly). 
Take a finite set of Borel probability measures $\mu_1, ... ,\mu_n$. Define
$\{\nu_1, ... ,\nu_s\} = \{\mu_i \circ T^k \ :\ i =1,... ,n,\
 k = 0, 1, ... , j-1\}$ (here $s = ij$). It follows from the condition of the
lemma  that for any positive $\epsilon_1$ and $\delta_1$ there exists 
$N = N(\epsilon_1, \delta_1)\in \N $ such that for all $i >N$
 \be\label{eq nu_l}
\nu_l(\{x \ : \ | f_i - f | > \epsilon_1   \}) <\delta_1, \quad l = 1,..., s.
\ee
 
For convenience, we introduce the following sets
 $$
 A_k(i, \epsilon_1) =  \{x \ : \ | f_i\circ T^k - f\circ 
T^k | > \epsilon_1\}, \quad  k = 0,..., j-1, 
$$ 
and
$$
C(i, \epsilon) = \{x \ : \ | a_i(j, x) - a(j,x) | > \epsilon \}.
$$
Denote 
$$
S(i, \epsilon) = \{x \ :\ \sum_{k=0}^{j-1} | f_i(T^k x) - f(T^kx) | > \epsilon.
$$
Since 
$$| a_i(j, x) - a(j,x) |  \leq \sum_{k=0}^{j-1} | f_i(T^k x) - f(T^kx) |,
$$
we see that  $C(i, \epsilon) \subset S(i, \epsilon)$. Take $\epsilon_1 =
\dfrac{\epsilon}{j}$ and $\delta_1 =\dfrac{\delta}{j}$; 
then it follows from the above definitions that 
$$
\bigcup_{k=0}^{j-1} A_k\left(i, \frac{\epsilon}{j}\right)
 \supset S(i, \epsilon).
$$
We need to prove that $\mu_l(C(i, \epsilon)) <\delta$ for all sufficiently 
large  $i$ and $l = 1,..., n$. Indeed, it follows from  \eqref{eq nu_l} that, 
for $i > N(\epsilon_1, \delta_1)$, 
$$
\ba
\mu_l(C(i, \epsilon)) \leq & \ \mu_l(S(i, \epsilon))  \\ 
\leq &\  \sum_{k=0}^{j-1} \mu_l\left(A_k\left(i, \frac{\epsilon}{j}\right)\right)\\
= & \ \sum_{k=0}^{j-1} \mu_l\circ T^k  \left(A_0\left(i, \frac{\epsilon}{j}
\right)\right)\\
< &\ j \frac{\delta}{j} = \delta.
\ea
$$
This proves the lemma.
\end{proof}

\begin{proposition}\label{prop density cbd} Let $a: \mathbb{Z} \times X 
\rightarrow G$ be a cocycle of an aperiodic $T \in Aut(X,\mathcal{B})$. 
Then there exists a sequence of coboundaries $(a_n)$ of $T$ such that 
$(a_n)$ converges to $a$ in the topology $\mathcal{T}$ (see Remark 
$\ref{equiv top}$ and Definition $\ref{def cocycle topology}$). 
\end{proposition}

\begin{proof} 
It is obvious that, for any cocycle $a: \mathbb{Z} \times X \rightarrow G$ of  $T \in Aut(X,\mathcal{B})$, there is a Borel function $f$ such that 
$a = a(f)$, i.e., 
\begin{equation}\label{eq a via f}
a(j,x) = \left\{\begin{matrix}
f(x)+ f(Tx)+...+f(T^{j-1}x), & \quad j\geq 1 \\ 
0, & \quad j = 0\\ 
-f(T^{-1}x) - f(T^{-2}x) - ... -f(T^{j}x), & \quad j\leq-1
\end{matrix}\right.
\end{equation}

In the proof,  we will use the notation introduced in this section above. 
By Proposition $\ref{lemma periodic approx}$, for every $T \in Aut(X,
\mathcal{B})$, there exists a sequence of periodic automorphisms $(P_i)$ 
of $(X,\mathcal{B})$ converging to $T$ in the topology $\tau$ (see
 Definition $\ref{def topology}$). It can be easily seen that $P_i$ and 
 $P_{i+1}$ agree (that is $P_ix = P_{i+1}x$ everywhere except on top of the
 $T$-towers $\xi_k(i)$ built over $A_i$ where $(A_i)$ is a vanishing 
 sequence of markers. Let $D_i$ denote the union of the top levels of 
 $T$-towers $\xi_k(i)$. Since $D_i \supseteq D_{i+1}$ and $\bigcap_i
 A_i = \emptyset$ , we see that   $\bigcap_i D_i = 
 \emptyset$.  Therefore, for every $x$, there exists a smallest number 
 $n(x)$ such that, for all 
 $i \geq n(x)$, $P_ix$ are all the same and equal to $Tx$. 

Next, we define  $K_j := \{x \in X : n(x) = j\}$, $j\in \mathbb{N}$.
 Note that  $K_j \subset K_{j+1}$ and $\bigcup_j K_j = X$. 
 Fix a finite set of probability measures $\mu_1,\mu_2, ... \mu_n
  \in \mathcal{M}_1(X)$ and take $\epsilon > 0$. Then there exists
 $j \in \mathbb{N}$, such that $\mu_l(K_j) > 1- \epsilon$ for $l = 1,2,...,n$. 

We recall that  the periodic automorphisms $P_i$ are taken from the 
full group $[T]$ and therefore the cocycle $a \in Z^1(\Gamma \times X, T)$
can be extended to $P_i$. This observation allows us to define
$$
f_n(x) := a (P_n, x), \quad \forall x \in X.
$$
By Lemma $\ref{periodic aut}$, every cocycle of $P_n$ is a
 coboundary. Hence there exists a sequence of Borel functions $g_n :X 
 \rightarrow G$ such that $f_n(x) = g_n(x) - g_n(P_nx)$. Moreover, 
recall that $P_n x= Tx$ for every $x \in K_n$. As a result, for every $x \in 
K_n$ we have  $f_n(x) = a (P_n, x) = a(T, x) = f(x)$. We further define a 
sequence of Borel  functions $F_n :X \rightarrow G$ as follows: 
$$
F_n(x)= g_n(x)- g_n(Tx), \quad \forall x \in X.
$$ 
By definition, the function $F_n$ is a $T$-coboundary for every $n$. 

lt remains to  show that $F_n \stackrel{\mathcal{T}}{\longrightarrow} f$ 
(see Definition $ \ref{def cocycle topology}$). For this, we prove that for 
every $\epsilon,  \delta >0$ there exists $n \in \mathbb{N}$ such that 
\begin{equation}\label{eq 4.2}
    \mu_l(\{x: |F_n(x) - f(x)| > \epsilon\}) < \delta \,,\,\,\, \forall \, l=1,2,...,n.
\end{equation}
Note that if $x \in K_n$, then $f_n(x) - f(x)$ and
$$
|F_n(x) - f(x)| = |g_n(P_nx) - g_n(Tx)| = 0.
$$ 
 Hence 
 $$ \mu_l(\{x: |F_n(x) - f(x)| > \epsilon\}) \subset X\setminus K_n, 
 \quad \forall \, l=1,2,...,n.
 $$ 
For every $\delta >0$, we can find $N$ such that for all $n \geq N$, 
$\mu_l(X \setminus K_n)< \delta$  for $ l = 1,2,...,n$, and then (\ref{eq 4.2}) follows. 

To finish the proof, we define the sequence  of  $T$-coboundaries $(a_n)$  
by  functions $F_n$ as in  \eqref{eq a via f}.  It follows from Lemma 
\ref{lem conv} that the converges of $(F_n)$ to the function $f$
in the topology $\mathcal T$ implies that $a_n(F_n) $ converges to 
$a(f)$ in $\mathcal T$.  It completes the proof.
\end{proof}

\textit{Proof of Theorem \ref{thm dense}} In light of Theorem 
$\ref{thm hyperfinite}$, Proposition  $\ref{prop density cbd}$ implies 
Theorem  $\ref{thm dense}$. \hfill{$\Box$}


\section{Cocycle over odometer action}\label{sec 5}

The goal of this section is to describe explicitly cocycles defined by $2$-odometers. In fact, the results of this section can be used for arbitrary 
uniquely ergodic Borel automorphisms since they are Borel isomorphic
to the 2-odometer. We will use the following definition of the 2-odometer
which is equivalent to Definition \ref{def odo}.
 
Consider the space $(X = \{0,1\}^\mathbb{N}, \mathcal{B})$, where 
$\mathcal{B}$ is the Borel sigma-algebra generated by cylinder sets. Let 
$\Gamma \subset  Aut(X,\mathcal{B})$ be the group of Borel
 automorphisms
generated by automorphisms $\langle \delta_1,   ..., \delta_n,.... 
\rangle$ where $\delta_n$ acts on $x = (x_i)\in X$ by the formula:
\begin{equation}\label{odometer}
    (\delta_n x)_i  = \left\{
        \begin{array}{ll}
            x_i & \quad i \neq n \\
            x_i + 1 \,\, (\rm{mod} \,\, 2)& \quad i = n. 
        \end{array}
    \right.
\end{equation}
We see that every $\delta_n$ is periodic, $\delta_n^2 = \mathbbm 1$, 
and any two generators $\delta_n$, $\delta_k$ commute.  
Obviously, the orbit equivalence relation $E_X(\Gamma)$ is 
\textit{hyperfinite} and preserves the product measure $\mu = 
\bigotimes_i \mu_i$ where $\mu_i(\{0\}) = \mu_i(\{1\}) = 1/2$. The  
group $\Gamma$ is orbit equivalent to the 2-odometer acting on  
$(\{0,1\}^ \mathbb{N}, \mathcal{B})$.

Cocycles over odometers have been extensively studied in ergodic 
theory. We refer, in particular,  to the papers \cite{Golodec_1969},  \cite{GolodetsSinelshchikov_1987} where the authors proved  several important results. Firstly, it was shown that
every cocycle is cohomologous to a cocycle that takes values in a 
countable subgroup $H$ of $G$, and, secondly, cocycles with dense
range are unique in the following sense: let $\alpha$ and $\beta$ be 
two cocycles with values in $G$ such that the skew products 
$\Gamma(\alpha)$ and $\Gamma(\beta)$ are ergodic, then 
 there exists an automorphism $R$ in the normalizer $N[\Gamma]$
  such that 
 $\alpha$ and $\beta\circ R$ are cohomologous (see Introduction).  

We use a similar approach  
 to prove the first result in the  setting of Borel dynamics. We do not
know whether the second result holds. 
We remark that for consistency with other parts of this paper
our proof is given for an abelian group $G$ though the same proof 
works for non-abelian groups. 

We reprove the following statement that was implicitly formulated
in \cite{Golodec_1969}. 

\begin{proposition} \label{prop Golodets}  Let the group $\Gamma
= \langle \delta_1, ... \delta_n, ...\rangle$ of Borel automorphisms of 
$\{0,1\}^{\mathbb N}$ be
defined as in \eqref{odometer}. Then for every cocycle $c : \Gamma 
\times X \to G$, there exists a sequence of Borel functions 
$(f_n : X \to G)_{n\in \N}$ such that 
\begin{equation} \label{eq_cocycle c}
\ba
c(\delta_n,x) =  & x_1f_1(\delta_n x) + ...+
x_{n-1} f_{n-1}(\delta_n x)  \\ 
+ & (-1)^{x_n} f_n(x) -x_{n-1} f_{n-1}(x)- ...-x_1 f_1(x),
 \ea
\end{equation}
where the function $f_n$  is invariant with respect to $ \delta_1, 
\delta_2, , ..., \delta_n $,   $n \in \mathbb{N}$.

Conversely, let $(f_n : X \to G)_{n\in \N}$, be a sequence of Borel maps such that each $f_n$ is invariant with respect to $\delta_1, \delta_2, , ..., \delta_n$. Then $(f_n)_{n \in \N}$ generates a  cocycle $c$ according to \eqref{eq_cocycle c}.
\end{proposition}

\begin{proof}
Since the transformations $\delta_i, i \in \N,$ are pairwise commuting,
 relation  \eqref{eq_cocycle c} can be extended to all $\gamma = 
 \delta_{i_1} \cdots \delta_{i_k} \in \Gamma$. First we show that if
there exist a sequence of functions $(f_n)$ with the invariance property
as described above, then $(\ref{eq_cocycle c})$  defines a cocycle of 
$\Gamma$. To do this, we show that  
$$ c(\delta_n \delta_k, x) = c(\delta_k \delta_n, x) \,\,\, \mathrm{and} \,\,\, c(\delta^2_n, x) = 0, \,\,\, \mathrm{for\, all} \,\, n,k \in \mathbb{N} \,\,\mathrm{and}\, x \in X.
$$ 
In other words, we need to prove that the definition of $c$ by 
\eqref{eq_cocycle c} gives the same result for two ways to compute
$c(\delta_n\delta_k, x)$.

By the cocycle identity, we have   $c(\delta_n \delta_k, x) = c(\delta_n,
 \delta_k x) + c(\delta_k, x)$. For definiteness,  we can assume that 
 $n > k$. In what follows, we will use the obvious property  
 $(\delta_k x)_i = x_i$ if $i\neq k$ and $(\delta_k x)_k =x_k +1 \  
 (\mathrm{mod}\  2)$. Then 
$$
\ba
 c(\delta_n, \delta_k x) = & x_1 f_1 (\delta_n \delta_k x) +... +
  (\delta_k x)_k f_k (\delta_n \delta_k x) + ...+ 
 x_{n-1}  f_{n-1} (\delta_n \delta_k x) \\ 
 + &  (-1)^{(\delta_k x)_n} f_n (\delta_k x) - 
 x_{n-1} f_{n-1} (\delta_k x)- ... -(\delta_k x_k) f_k (\delta_k x) - ... -
  x_1 f_1 (\delta_k x).
\ea
$$ 
 Using the fact that, for each $i\in \N$, the function $f_i$ is invariant 
with respect to $ \delta_1,  \delta_2, , ..., \delta_i $, we get
$$
\ba
 c(\delta_n,  \delta_k x) = & x_1 f_1 (\delta_n \delta_k x) +... + 
 (\delta_k x)_k f_k (\delta_n x) + ...+ x_{n-1} f_{n-1} (\delta_n x) + 
 (-1)^{(\delta_k x)_n} f_n (x) \\ 
 & - x_{n-1} f_{n-1} (x) - ... -(\delta_k x)_k f_k (x) -x_{k-1} f_{k-1}
  (\delta_k x)- ... - x_1 f_1 (\delta_k x).
\ea
$$ 

Similarly, we have by \eqref{eq_cocycle c}
$$
\ba
c(\delta_k, x) =&  x_1 f_1 (\delta_k x) + ... + x_{k-1} f_{k-1} (\delta_k x) + (-1)^{x_k} f_k (x)\\
& - x_{k-1} f_{k-1} (x)  -...-x_1 f_1 (x).
\ea
$$
After taking the sum and simplifying, we obtain that 
\begin{equation}\label{eq 1.3}
\ba
c(\delta_n \delta_k, x)
= & x_1 f_1 (\delta_n \delta_k x) + ... + (\delta_k x)_k f_k(\delta_n x)
+ ... +x_{n-1}f_{n-1}(\delta_nx) \\
& +   (-1)^{(\delta_k x)_n} f_n (x) - x_{n-1} f_{n-1} (x)
- ... -(\delta_k x)_k f_k (x)\\
& + (-1)^{x_k} f_k (x) - x_{k-1} f_{k-1} (x) -...-x_1 f_1 (x).
\ea
\end{equation} 

Next, we represent $ c(\delta_k \delta_n, x)$ as  $c (\delta_k, 
\delta_n x) + c(\delta_n, x)$ and compute noticing that 
$(\delta_n x)_k = x_k$:
$$
\ba
c(\delta_k, \delta_n x) = & x_1 f_1 (\delta_k \delta_n x) + ...+ (-1)^{x_k} f_k (\delta_n x)\\
& - x_{k-1} f_{k-1} (\delta_n x) -...- x_1 f_1 (\delta_n x)
\ea 
$$ 
and
$$
\ba
c(\delta_n, x) = & x_1 f_1 (\delta_n x) +...+ x_{k-1} f_{k-1} (\delta_n x) + x_k f_k (\delta_n x) + ...\\
& + (-1)^{x_n} f_n (x)  - x_{n-1} f_{n-1} (x) -...\\
& - x_{k+1} f_{k+1} (x) - x_k f_k (x)- x_{k-1} f_{k-1} (x) -...- 
x_1 f_1 (x). 
\ea
$$ 
Thus, we get
\be\label{eq_cocycle kn}
\ba
c (\delta_k, \delta_n x) + c(\delta_n, x) =  & x_1 f_1 (\delta_k \delta_n x) +... + (-1)^{x_k} f_k (\delta_n x) + x_k f_k (\delta_n x) + ...\\
& + (-1)^{x_n} f_n (x)  - x_{n-1} f_{n-1} (x) -...\\
&- x_k f_k (x) - x_{k-1} f_{k-1} (x) -...- x_1 f_1 (x). 
\ea
\ee
One can easily see (by considering all possible values for $x_k$)
 that  the following relations hold:
$$
(\delta_kx)_k f_k(\delta_n x) =  (-1)^{(\delta_n x)_k} f_k(\delta_n x)
+ x_k f_k(\delta_n x)
$$
and 
$$
- x_k f_k(x) = (\delta_kx)_k f_k( x) + (-1)^{x_k} f_k(x).
$$ 
Comparing $(\ref{eq 1.3})$ and $(\ref{eq_cocycle kn})$, we conclude
that $c(\delta_n, \delta_k x) + c(\delta_k, x)  = c(\delta_k, \delta_n x)
+c(\delta_n x)$ for all distinct  integers $ n, k$.
 
To see that, for every $n \in \N$, the cocycle $c$ has the property 
$c(\delta^2_n, x) = 0$, we observe
$$
\ba
c(\delta_n, \delta_n x) + c(\delta_n,x) = & (\delta_nx)_1 
f_1(\delta^2_n x) + ...
+ (-1)^{(\delta_n x)_n} f_n(\delta_n x) - \cdots \\
& - (\delta_nx)_{n-1} f_{n-1} (\delta_n x) -  \cdots - (\delta_nx)_{1}
 f_{1} (\delta_n x) \\
& + x_1 f_1 (\delta_n x) +...+ (-1)^{x_n} f_n (x) - ...- x_1 f_1 (x). 
\ea
$$ 
Because $\delta^2_n = \mathbbm 1$ and $f_n$ is $\delta_n$-invariant,
we see that 
$$
c(\delta^2_n,x) = (-1)^{(\delta_n x)_n} f_n (x) + (-1)^{x_n} f_n (x) = 0. 
$$ 
This proves that relation $(\ref{eq_cocycle c})$ defines a cocycle of the
group $\Gamma$. 

Conversely, if a cocycle $c$ is given, then the functions $f_n$ are
determined as follows:
set $f'_n(x) = c(\delta_n, x)$ for $x$ from the cylinder set 
$A_n(0,..., 0)$
 generated by the first $n$ zeros.  Then $f_n'$ is extended on $X$ by invariance
 with respect to the subgroup $\langle \delta_1, ... \delta_n\rangle$ to obtain 
 the function $f_n$. 
 \end{proof}

Let $\alpha$ and $\beta$ be two cocycles of $\Gamma$, which are
 determined as in Proposition \ref{prop Golodets} by 
sequences of Borel functions $f_n : X \rightarrow G$ and $\overline{f_n}
 : X \rightarrow G$, respectively. Define two new sequences of functions
$\psi_n : X \rightarrow G$ and  $\overline{\psi_n} : X \rightarrow G$
 as follows:
\begin{equation}\label{5b}
    \psi_n(x) = -x_n f_n (x) -x_{n-1} f_{n-1} (x) - ... -x_1  f_1(x) 
\end{equation}
\begin{equation}\label{5c}
    \overline{\psi_n}(x) = -x_n\overline {f_n} (x) -x_{n-1} \overline{f_{n-1}}(x) - \cdots -x_1 \overline{f_1}(x) 
\end{equation} We denote by $\{W_i\}_{i=1}^{\infty}$ a system of neighborhoods of  $0 \in G$ with the following properties: 

$(i)$ $W_i$ is compact for every $i$;

$(ii)$ $W_i$ is symmetric for every $i$ (i.e. $W_i = -W_i$);

$(iii)$ $W_{i+1} + W_{i+1} \subset W_{i}, i \in \N$. 

\begin{proposition}\label{lemma 5.1} Let $\alpha$ and $\beta$ be
two cocycles of the group $\Gamma$ with values in a l.c.s.c. group $G$.
Let $(f_n)$ and $(\overline{f_n})$ be the sequences of functions 
determined by $\alpha$ and $\beta$, respectively,
  according to Proposition \ref{prop Golodets}.
Assume  that, for all $x \in X$ and  $n \in \mathbb{N}$, 
$$f_n (x) - \overline{f_n} (x) \in W_n$$
where the neighborhoods $(W_n)$ satisfy conditions (i) - (iii). 
Then the cocycles $\alpha$ and $\beta$ are cohomologous. 

\end{proposition}

\begin{proof}

Define a sequence of functions $g_n (x) := -\psi_n(x) + \overline{\psi_n}(x), \, n \in \N$ where $\psi_n$ and 
$\overline{\psi_n}$ are as in \eqref{5b} and \eqref{5c}.   
Thus for all $n, k \in \mathbb{N}$, we have 
$$
g_{n+k} (x) - g_n (x) = 
 -\psi_{n+k}(x) + \overline{\psi_{n+k}}(x) +\psi_n(x) - \overline{\psi_n}(x).
$$
It follows from \eqref{5b} and \eqref{5c} that
$$
-\psi_{n+k}(x) =  x_{n+k} f_{n+k} (x) +  ... 
+ x_{n+1} f_{n+1}(x)  + \psi_n(x),
$$
and a similar formula holds for $\overline{\psi_{n+k}}(x)$. Hence, 
$$
\ba
g_{n+k} (x) - g_n (x)  = & x_{n+k} f_{n+k} (x) - x_{n+k} 
\overline{f_{n+k}} (x) +....\\
& + x_{n+1} f_{n+1}(x) - x_{n+1} \overline{f_{n+1}}(x).
\ea
$$
It follows from the condition of Proposition that   $f_{n+i} (x) - 
\overline{f_{n+i}} (x) \in W_{n+i} $ for all $i,n \in 
\mathbb{N}$. Hence we have 
$$ 
x_{n+i} f_{n+i} (x) - x_{n+i} \overline{f_{n+i}} (x)\in W_{n+i}, 
\ \ \ \forall \, i,n \in \mathbb{N}
$$ 
By the choice of $W_i$, we obtain 
\begin{align*} 
g_{n+k} (x) - g_n (x) \in &\  W_{n+k} + W_{n+k-1} + ....+ W_{n+1}. 
\\ 
 \subset &\ W_{n+k-1} + W_{n+k-1}  + ....+ W_{n+1}. \\
 \subset &\ W_{n+k-2}  + ....+ W_{n+1}. \\
& \cdots \cdots  \cdots \cdots \cdots \cdots\\
 \subset &\  W_n.
\end{align*} 
Using the Cauchy criterion, there exists a Borel function $g : X \rightarrow G$ such that, $g_n$ converges uniformly to $g$ on $X$. 

Without loss of generality, we can assume that $n\geq k$. 
Since $(\delta_n x)_i = x_i$ where $i =1 ,..., k-1,$ and 
$$-\overline{\psi_{k-1}}(\delta_k (x)) = x_1\overline{f_1}(\delta_kx)  + ...+x_{k-1} \overline{f_1}(\delta_kx)
$$ 
$$
\overline{\psi_{k-1}} (x) = -x_{k-1} \overline{f_{k-1}}(x)+ ...+ -x_1 \overline{f_1}(x),$$
we can  compute $\widehat{\beta}(\delta_k, x) = g_n(\delta_k x) + \beta (\delta_k, x) - g_n (x)$ as follows:
$$
\ba
\widehat{\beta}(\delta_k, x) & =   -\psi_n(\delta_k x) + \overline{\psi_n}(\delta_k x) + x_1\overline{f_1}(\delta_kx)  + ...+x_{k-1} \overline{f_1}(\delta_kx)
+ (-1)^{x_k} \overline{f_k}(x) \\
& \ \ -x_{k-1} \overline{f_{k-1}}(x) - ... - x_1 \overline{f_1}(x) - 
( -\psi_n(x) + \overline{\psi_n}(x))\\
& = -\psi_n(\delta_k x) + \overline{\psi_n}(\delta_k x)  -\overline{\psi_{k-1}}(\delta_k (x)) + (-1)^{x_k} \overline{f_k}(x)
+  \overline{\psi_{k-1}} (x)\\
&\ \ +\psi_n(x) - \overline{\psi_n}(x)\\
& =  -\psi_n(\delta_k x) -x_n \overline{f_n} (\delta_k x) - x_{n-1} \overline{f_{n-1}} (\delta_k x) - ... - (\delta_k x)_k \overline{f_k} (\delta_k x) \\
& \ \ + (-1)^{x_k} \overline{f_k}(x) 
+ x_n \overline{f_n} (x)+...+ x_k \overline{f_k} (x) + \psi_n (x).
\ea
$$
Since $n \geq k$, the function $f_n$ is invariant with respect to 
$\delta_1, ...,\delta_k$, we have
$$
\ba
\widehat{\beta}(\delta_k, x) & = -\psi_n(\delta_k x) -x_n \overline{f_n} 
(x) - x_{n-1} \overline{f_{n-1}} (x) - ... - (\delta_k x)_k 
\overline{f_k} (x) + (-1)^{x_k} \overline{f_k}(x)\\
& \ \ +  x_k \overline{f_k} (x)+...+ x_n \overline{f_n} (x) + \psi_n (x). 
\ea
$$
 After simplifying, we obtain that
$$
\widehat{\beta}(\delta_k, x) = -\psi_n(\delta_k x) - (\delta_k x)_k 
\overline{f_k} (x) + (-1)^{x_k} \overline{f_k} (x) + x_k 
\overline{f_k} (x) + \psi_n (x). 
$$
It remains to show that 
\be\label{eq identity}
(\delta_k x)_k \overline{f_k} (x) + (-1)^{x_k} \overline{f_k} (x) + 
x_k \overline{f_k} (x) = 0.
\ee
Indeed, if $x_k = 0$, then $x_k \overline{f}_k (x)=0$, and 
$(\delta_k x)_k =1 $ implies that 
$- (\delta_k x)_k \overline{f_k} (x) = - \overline{f_k} (x) $. If $x_k =1$, then  then $- (\delta_k x)_k \overline{f_k} (x) = 0$ and 
$(-1)^{x_k} \overline{f}_k (x) = - \overline{f}_k (x)$ .  Thus in both cases we get 
$$ 
g_n(\delta_k x) + \beta (\delta_k, x) - g_n (x) = -\psi_n(\delta_k x) + \psi_n (x). $$
On the other hand, 
$$
\ba 
-\psi_n(\delta_k x) + \psi_n (x) & = x_n f_n (\delta_k x)  +..
 + (\delta_k x_k) f_k (\delta_k x) + ... \\ 
 & \ \ + x_1 f_1 (\delta_k x) -x_n f_n (x) -  ...- x_k f_k (x) -...- x_1 
 f_1 (x). 
\ea 
$$
By invariance of $f_n$ with respect to of $\delta_1, ...,\delta_k$, we  
can write down the above equality as
\be\label{eq for alpha}
\ba 
-\psi_n(\delta_k x) + \psi_n (x) & =  x_n f_n (x)  +... + (\delta_k x)_k
 f_k (x) + x_{k-1} f_{k-1} (\delta_k x) + ... \\
 & \ \ + x_1 f_1 (\delta_k x)  -x_n f_n (x) -  ...- x_k f_k (x) - ...- 
 x_1 f_1 (x)\\
& = - \psi_{k-1}(\delta_k x) + (-1)^{x_k} f_k (x) + \psi_{k-1} (x)\\
&  = \alpha (\delta_k, x).
\ea
\ee
The first equality in \eqref{eq for alpha} is due to relation 
\eqref{eq identity}, applied to the function $f_k$, and the second 
equality is, in fact, a short form of the definition of $\alpha$.

Thus,  we proved that, for every $n \geq k$ and all $x\in X$,
$$
g_n(\delta_k x) + \beta (\delta_k, x) - g_n (x) = \alpha (\delta_k, x) .
$$ 
Since $g_n(x) \to g(x)$ as $n \rightarrow \infty$,  we conclude that
$$
g(\delta_k x) + \beta (\delta_k, x) - g (x) = \alpha (\delta_k, x).
$$ 
Because the group $\Gamma$ is generated by $\delta_k, k \in \N,$,
we see that  the cocycles $\alpha$ and $\beta$ are cohomologous. 
\end{proof} 

\begin{theorem}\label{thm 5.1} Let $\Gamma$ be a free group of Borel automorphisms which is orbit equivalent to the 2-odometer. Let $\alpha$ be a $\Gamma$-cocycle with values in a l.c.s.c. group $G$ and $H$ a 
dense countable subgroup of $G$.  Then  the cocycle $\alpha:\Gamma \times X \rightarrow G$ is  cohomologous to a  cocycle $\beta $ with values the subgroup $H$.

\end{theorem}
\begin{proof} Without loss of generality, we can consider cocycles
of the 2-odometer. By Proposition \ref{prop Golodets} 
the cocycle $\alpha$ is determined by the functions $f_n : X \rightarrow G, n \in \N$. Take a sequence of symmetric neighborhoods of 
0 in $G$ which satisfies the properties (i) - (iii) (see above).
 Approximate each function $f_n (x)$ by a function 
$\overline{f_n} (x)$ with values in $H$ so that 
$f_n (x) - \overline{f_n} (x) \in W_n $ for each $x \in X$, and 
additionally, 
$\overline{f_n} (\delta_j x) = \overline{f_n} (x)$, for $1 \leq j \leq n$. 
Clearly it can be done because the functions $f_n$ have this property.

Hence, we satisfy the conditions of Proposition $\ref{lemma 5.1}$. 
Construct the $\Gamma$-cocycle $\beta$  which is determined
 by the sequence of functions $\overline{f_n} (x)$, then $\beta$ is 
 cohomologous to  $\alpha$. 
\end{proof}

\section{Borel version of Gottschalk-Hedlund theorem}\label{sec 6}
The following is a version of the Gottschalk-Hedlund (G-H) theorem for Borel
automorphisms. Our proof is a modification of the proof of 
Gottschalk-Hedlund theorem given by F. Browder \cite{Browder_1958}. 

We will consider homeomorphisms of a Polish space. It is well known 
that every Borel automorphism admits a continuous model, i.e., it is
Borel isomorphic to a homeomorphism of a Polish space, see
e.g. \cite{Kechris_1995}.  
We say that a homeomorphism $T \in Aut (X,\mathcal{B})$ acting on 
a Polish space $X$  is \textit{minimal} if every $T$-orbit is dense in $X$, 
i.e., for every $x \in X$, $\overline{\{T^i x : i \in \Z \}} = X$. 
 There exist Polish spaces that  admit minimal 
homeomorphisms (we thank \cite{Snoha_2019} for examoples of such spaces). 

We note that in Theorem $\ref{thm GH}$ we consider \textit{bounded} cocycles of homeomorphism of a Polish space, while G-H theorem for topological dynamics (see \cite{Gottschalk_Hedlund_1955}) has no such restriction. This is due to the fact that the underlying space in Theorem 
$\ref{thm GH}$ is a non-compact Polish space. In topological dynamics \textit{continuous} cocycles of homeomorphism of a compact space are studied. Here we study \textit{Borel} cocycles of homeomorphisms of a non-compact Polish space. Hence, we have to limit our discussion to bounded cocycles. We do not know whether the result holds without 
this assumption. 

In the proof of Theorem \ref{thm GH} we will use the following fact: Every locally compact second countable group $G$ has a left-invariant metric $d$ which is proper, that is every closed $d$-bounded set in $G$ is compact (see \cite[Theorem 2.B.4]{Cornulier_Harpe_2016}).

\begin{theorem}\label{thm GH} Let $(X,\mathcal{B})$ be a Polish 
space and $T \in Aut(X,\B)$ is a minimal homeomorphism of 
$(X,\mathcal{B})$. 
Let $h : X \rightarrow G$ be a bounded Borel map from $X$ to a l.c.s.c.
 abelian group $G$. Then, the function $h$ is a 
coboundary (i.e., there exists a bounded Borel 
function $f : X \rightarrow G$  such that $f(T x)- f(x)=h(x)$, 
$x \in X$), if and only if there exists $M>0$ such that
\begin{equation*}
    \underset{x \in X}{\mathrm{Sup}}\,\, \left|\underset{k=-j}{\overset{j}{\sum}} h(T^k x)\right| \leq M,
\end{equation*} for all $j \geq 0$. 
\end{theorem}

Before we begin to prove Theorem $\ref{thm GH}$, we define
some maps and prove Lemmas \ref{lemma 6A} - \ref{lemma 6C}. Let $\psi : X \times G \rightarrow G$, as $\psi (x,g) = g + h(x)$ where $h(x)$ is the Borel map as in the statement of 
Theorem \ref{thm GH}. Next, we define the skew product $X \times G \rightarrow X\times G$ as $\pi (x, g) = (Tx,\psi (x,g)) = (T x, g+ h(x))$.

Denote by $Orb_\pi(x,g) = \bigcup _{n \in \mathbb{Z}} 
\{\pi^n (x,g)\}$ the orbit of $(x,g)$ under $\pi$ and by $F(x,g)=$ $ 
\overline{Orb_\pi(x,g)}$ the orbit closure in $X \times G$. Let $p_X$
and $p_G$ denote the natural projections from $X \times G$ to $X$ and 
$G$, respectively. We assume that for each point $(x,g) \in X \times G$ the set $p_G (F(x,g))$ is contained in a compact subset of $G$. 

\begin{remark} \label{bounded assumption} We note that  the condition
that  $\underset{k=-j}{\overset{j}{\sum}} h(T^k x)$ is bounded in 
$G$ for all $x \in X$ and $j \geq 0$ is equivalent to the fact that the 
orbit (with respect to $\pi$) of any point $(x,g) \in X \times G$ has a
bounded and hence a precompact image in $G$ under the projection 
map $p_{G}$ of $X \times G$ into $G$. This in turn implies that 
$p_G (F(x,g))$ is contained in a compact subset of $G$. 

\end{remark}

Consider the family $J$ of subsets $F$ of $X \times G$ such that
\begin{multline*}
J = \{F \ |\ F\ \mathrm{is\,\,nonempty\,\,closed \,\, subset \,\, of}\,\, 
X \times G \,\,; (x,g) \in F \,\, \mathrm{implies \,\, that}\,\,\\ \pi(x,g) 
\in F ;\, p_G(F)\,\, \mathrm{is\,\, contained \,\, in\,\,a\ compact \,
\,subset\,\, of} \,\, G\}.
\end{multline*} 
Obviously, $J$ is nonempty since, for any point $(x_0, g_0) \in X 
\times G$ the set $F(x_0, g_0)$ is  in $J$.

\begin{lemma}\label{lemma 6A}If $F \in J$, then $p_X(F) = X$.

\end{lemma}

\begin{proof} Let $(x_0,g_0) \in F$. Since $\pi^n (x_0,g_0) 
\in F$, $p_X (\pi^n (x_0,g_0)) \in p_X(F)$. Thus $p_X(F)$ contains the 
dense set $\{T^k (x_0)\}$. Hence  $p_X(F)$ is dense in $X$.

Next, for $F \subset X\times G$, we have $F \subset X \times 
\overline{p_G(F)}$ and $\overline{p_G(F)}$ is a compact set in $G$. 
Since the projection $p_X(F)$ is a closed map, we obtain that  
$ p_X(F)$ is closed in $X$. We showed that  $p_X(F)$ is dense and 
 closed in $X$, hence $p_X(F)= X$. 
\end{proof}

\begin{lemma}\label{lemma 6B} The family of sets $J$ has a minimal
 element under inclusion. Every orbit closure $F(x,g)$,  $(x, g) \in X 
 \times G$,  contains a minimal element of $J$. 

\end{lemma}

\begin{proof}
We use Zorn's lemma. Consider  a totally ordered (with respect to 
 inclusion) chain $\{F_\alpha\}$ in $J$. Let $F_0 = 
 \bigcap_\alpha F_\alpha $. Then $F_0$ is a closed $\pi$-invariant set
 and $p_G(F_0)$ is clearly contained in a compact set of $G$. To prove 
 that $F_0 \in J$, we show $F_0  \neq \emptyset$.  

Let $x_0 \in X$, consider $G_\alpha = F_\alpha \cap p_X^{-1} (x_0)$. 
By Lemma \ref{lemma 6A},  $p_X(F_\alpha) = X$,  therefore 
$G_{\alpha}$ is a nonempty closed subset  for any $x_0\in X$ and 
any $\alpha$.  
Moreover ,$G_\alpha \subset x_0 \times \overline{p_G (F_\alpha)}$. 
We note that $ x_0 \times \overline{p_G (F_\alpha)}$ is compact since 
it is mapped homeomorphically by $p_G$ to a compact set 
$ \overline{p_G (F_\alpha)}$. Since $G_\alpha$ is  compact for each $\alpha$, $G_0 = \bigcap_\alpha G_\alpha $ is non-empty. 
Since $G_0 \subset F_0$ we conclude that  $F_0$ is non-empty. 
\end{proof} 

Let $\xi : G \rightarrow G$ be a homeomorphism of $G$ such that it commutes with $\psi$ i.e. $\psi(x, \xi g) = \xi \psi (x,g)$ for all $x \in X$ and $g \in G$. Let  $S_{\xi} : X \times G \rightarrow X \times 
G $ be a homeomorphism defined by $S_{\xi} (x,g)= (x, \xi g) = (x, \xi g)$.  

\begin{lemma}\label{lemma 6C} Let $F_0$ be a minimal element of 
$J$ and suppose that for a fixed point $x_0 \in X$, the points 
$(x_0, g_0)$, $(x_0, g_1)$ lie in $F_0$. Suppose further that there exists a homeomorphism $\xi$ of $G$ onto itself such that it commutes with $\psi$ and $\xi (g_0) = g_1 $. Then $S_{\xi_k} F_0 = F_0$.
\end{lemma}

\begin{proof}
Since $\xi$ commutes with $\psi$, we get 

$$
\ba
S_{\xi} \pi (x_0, g_0) = S_{\xi} (T x_0, \psi (x_0, g_0)) = (T x_0, \xi \psi (x_0, g_0))
\\
= (T x_0, \psi (x_0, \xi g_0)) = \pi (x_0, \xi g_0) = \pi S_{\xi}(x_0, g_0).
\ea
$$

Thus $S_{\xi} \pi ^n = \pi^n S_{\xi}$, i.e. $S_{\xi} (Orb_{\pi} (x_0, g_0)) = Orb_{\pi}(x_0,\xi g_0) $. Using the fact that  $S_{\xi}$ is a homeomorphism we get $S_{\xi} F(x_0, g_0) = F(x_0, \xi g_0)$. Since $F_0$ is a 
minimal element of $J$, by assumption it contains both $(x_0,g_0)$ and $(x_0, g_1)$  we get $F_0 = F(x_0, g_0) = F(x_0,g_1)$. But $S_{\xi} F_0 = S_{\xi} F(x_0, g_0) = F(x_0,\xi g_0) =  F(x_0, g_1) = F_0$. 

\end{proof}

\textit{Proof of Theorem \ref{thm GH}}. Let $B(0,r)$ denote the ball 
of radius $r$ centered at $0 \in G$ with respect to a translation 
invariant metric on $G$. 
We first  assume that there exists a bounded Borel function $f: X 
\rightarrow G$ such that $f(x) \in B(0,m)$ for some $m >0$, and  
$h(x) = f(T x) - f(x)$ for all $x \in X$. Then, it is clear that 
 $$ 
 \underset{k=-j}{\overset{j}{\sum}} h(T^k x) =  -f(T^{-j} x) + 
 f(T ^{(j+1)} x)  \in B(0, 2m).
 $$
 Hence,  $\underset{k=-j}{\overset{j}{\sum}} 
  h(T^k x)$ is bounded in $G$ for all $x$.

Conversely, assume that, for all $x \in X$ and for all $j \geq 0$, 
$\underset{k=-j}{\overset{j}{\sum}} h(T^k x)$ is bounded in $G$.
 Thus, for any point $(x_0, g_0) \in X \times G$, the set 
 $p_G(F(x_0,g_0))$ is contained in a compact set of $G$ (see Remark 
 \ref{bounded assumption}). Therefore, we can apply Lemmas 
 \ref{lemma 6A} - \ref{lemma 6C}. 

Let $F_0$ be a minimal closed invariant set in $X \times G$ with respect 
to $\pi$. We will show that, for any $x_0 \in X$, $F_0$ contains at 
most one point of the form $(x_0,g)$. To see this, assume that for some 
$x_0 \in X$, the set $p_X^{-1} x_0 \cap F_0 $ contains two distinct 
points $(x_0, g_0)$ and $(x_0, g_1)$. Let $k = g_1 - g_0$; then the
 map $\xi_k (g) = g + k$ is a 
homeomorphism of $G$ onto itself which commutes with $\psi$,  
and $\xi_k(g_0) =  g_1$. By Lemma \ref{lemma 6C}, $S_{\xi_k} F_0 
= F_0$ where $S_{\xi_k} (x,g)= (x, g+k)$. Hence, $S_{\xi_k}^i F_0
 = F_0$ for any integer $i$. This contradicts the boundness of 
 $p_G(F_0)$. Thus, $F_0$ has at most one point $(x_0, g_0)$ for 
arbitrary $x_0 \in X$. Therefore, we can uniquely define a function 
$f: X \rightarrow G$ by the condition $f(x_0) = g_0$ where $(x_0, g_0)
 \in F_0$. By Lemma \ref{lemma 6A}, the function $f$ is defined at
  every point  of  $X$. 
Moreover, $f$ can also be considered as a function on $X$ with values in 
the compact set $\overline{p_G (F_0)}$.

Recall following result: If $Y$ is a topological space, $Z$ a compact
 space, and $s: Y \rightarrow Z$ is a function,  then the graph of $s$ is 
 closed if and only if $s$ is continuous.

Since the set $F_0$ is the graph of $f$ and $F_0$  is closed, we 
conclude that $f$ is a continuous function. Finally, for 
$\pi (x_0, f(x_0)) \in F_0$, we have $(T x_0, f(x_0) + h(x_0)) \in F_0$. 
Thus, by definition of $f$, we get $f(T x_0) = f(x_0) + h(x_0)$ as
 needed. \hfill{$\Box$}
\\

\textbf{Acknowledgments} The authors are very thankful to Palle 
Jorgensen and Paul Muhly for their interest to this work and useful 
discussions. We thank the referee for the careful reading and valuable 
suggestions. 
\bibliographystyle{alpha}
\bibliography{references1.bib}

\end{document}